\documentclass[12pt,a4paper]{article}
\usepackage[centertags]{amsmath}
\usepackage{amsfonts}
\usepackage{amssymb}
\usepackage{amsthm}
\usepackage{color}
\usepackage{marginnote}
\usepackage{float}

\usepackage[all]{xy}
\usepackage{fullpage}
\usepackage{tikz}
\usepackage{verbatim}
\usetikzlibrary{arrows}

\newtheorem{coro}{Corollary}
\newtheorem{defi}{Definition}
\newtheorem{ex}{Example}
\newtheorem{rem}{Remark}
\newtheorem{prop}{Proposition}
\newtheorem{lem}{Lemma}

\newtheorem{theo}{Theorem}

\newcommand{\fFor}{\bf{fFor}}

\newcommand{\MTL}{\mathcal{MTL}}
\newcommand{\fMTL}{f\mathcal{MTL}}
\newcommand{\faMTL}{fa\mathcal{MTL}}
\newcommand{\faMTLc}{fa\mathcal{MTL}c}
\newcommand{\GfMTL}{r\mathcal{MTL}}
\newcommand{\fLF}{f\mathcal{LF}}
\newcommand{\ufLF}{u\mathcal{LF}}
\newcommand{\Skel}{\mathfrak{S}}
\newcommand{\nzid}{\mathcal{J}(\mathcal{I}(M))}
\newcommand{\Ida}{\mathcal{G}}
\newcommand{\Vuelta}{\mathcal{H}}
\newcommand{\fCoh}{f\textbf{Coh}}

\newcommand{\Shv}{\mathbf{Shv}}
\newcommand{\Set}{\textbf{Set}}
\newcommand{\Pprod}{\mathcal{P}}

\newcommand{\GoodId}{\text{local unit}}
\newcommand{\Appropiate}{representable}

\newcommand{\dec}[1]{\mathcal{D}(\textbf{#1})}
\newcommand{\filter}[1]{X_{#1}}
\newcommand{\JI}[1]{\mathcal{J}(\mathcal{I}(#1))}
\newcommand{\etiq}[1]{l_{#1}}

\newcommand{\Filter}[2]{X_{#1}^{#2}}
\newcommand{\cov}[2]{\mathcal{C}_{#1}(#2)}

\newcommand{\mtl}[2]{M_{#1}(#2)}

\usepackage{color}

\begin{document}

\title{On finite MTL-algebras that are representable as poset products of archimedean chains}

\author{Jos\'{e} Luis Castiglioni and William Javier Zuluaga Botero}


\maketitle

\begin{abstract}
\noindent
We obtain a duality between certain category of finite MTL-algebras and the category of finite labeled trees. In addition we prove that certain poset products of MTL-algebras are essentialy sheaves of MTL-chains over Alexandrov spaces. Finally we give a concrete description for the studied poset products in terms of direct products and ordinal sums of finite MTL-algebras.
\end{abstract}

\section*{Introduction}
In \cite{EG2001} Esteva and Godo introduced MTL-logic as the basic fuzzy logic of left-continuous t-norms. Furthermore, a new class of algebras was defined, the variety of MTL-algebras. This variety constitutes an  equivalent algebraic semantics for MTL-logic. MTL-algebras are essentially integral commutative residuated lattices with bottom satisfying the prelinearity equation:
$$
(x \to y) \vee (y \to x) \approx 1
$$

\noindent
This paper is divided as follows. Section \ref{Preliminaries} is devoted to present the basic contents that are necessary to understand this work. In Section  \ref{section Finite archimedean MTL-chains}, we characterize the finite archimidean MTL-chains in terms of their nontrivial idempotent elements. In Section \ref{section Finite labeled forests}, we show that there exist a functor from the category of finite MTL-algebras to the category of finite labeled forests. We take advantage of the intimate relation between idempotent elements and filters, that is given for the case of finite MTL-algebras. In Section \ref{section Forest Product of MTL-algebras}, we study the forest products of MTL-chains. We prove that such construction is, in fact, a sheaf over an Alexandrov space whose fibers are MTL-chains. In Section \ref{section from finite forest products to MTL-algebras} we use the results obtained in Section \ref{section Forest Product of MTL-algebras} in order to establish a functor from the category of finite labeled forest to the category of finite MTL-algebras. We also bring a duality theorem between the category of representable finite MTL-algebras and finite labeled forest. Finally, we present a description of the forest product of finite MTL-algebras in terms of ordinal sums and direct products of finite MTL-algebras.

\section{Preliminaries}
\label{Preliminaries}

The aim of the following section is to give a brief survey about the background on MTL-algebras required to read this work. We present some known definitions and some particular constructions for semihoops that naturally can be extended to MTL-algebras.
\\

\noindent
We write $\Set$ to denote the category whose objects are sets and their morphisms are set functions.
\\

\noindent
A \emph{semihoop}\footnote{Some authors (for example \cite{NEG2005}) name prelinear semihoop what we call here simply semihoop.} is an algebra $\textbf{A}=(A,\cdot, \rightarrow, \wedge, \vee, 1)$ of type $(2,2,2,2,0)$ such that $(A,\wedge,\vee)$ is lattice with $1$ as greatest element, $(A,\cdot,1)$ is a commutative monoid and for every $x,y,z\in A$ the following conditions holds:
\\

\begin{tabular}{cccccc}
\emph{(residuation)} & & & &  & $xy\leq z\; \text{if and only if}\; x\leq y\rightarrow z$
\\
\emph{(prelinearity)} & & & & & $(x\rightarrow y)\vee (y\rightarrow x) = 1$
\end{tabular}
\\

\noindent
Equivalently, a semihoop is an integral, commutative and prelinear residuated lattice.
\\

\begin{rem}
It is customary in the literature on semihoops, \cite{NEG2005}, to present them in the signature $(\cdot, \rightarrow, \wedge, 1)$, beeing $\vee$ a defined operation.
However, although less common, it is also possible to present them in the signature $(\cdot, \rightarrow, \vee, 1)$, now beeing the $\wedge$ defined as
\begin{equation}
\label{inf in semihoops}
x \wedge y := (x \cdot (x \to y)) \vee (y \cdot (y \to x)).
\end{equation}

\noindent Let us check that the operation defined in \eqref{inf in semihoops} is the imfimum in $H$.

\noindent On one hand, since $x \cdot (x \to y) \leq x \cdot 1 = x$ and $y \cdot (y \to x) \leq x$, we get $x \wedge y \leq x$. Similarly, we deduce that $x \wedge y \leq y$.

\noindent On the other hand, if we assume that $c \leq x, y$, by monotonicity, we get that $x \cdot (x \to y) \geq c \cdot (x \to y)$ and $y \cdot (y \to x) \geq c \cdot (y \to x)$. Hence, $x \wedge y \geq (c \cdot (x \to y)) \vee (c \cdot (y \to x)) = c \cdot ((x \to y) \vee (y \to x))$. Since, by prelinearity, the rightmost term of this inequality is $c$, we get that $x \wedge y \geq c$.
\end{rem}

\noindent
This makes $(H, \cdot, \to, \wedge, \vee, 1)$ an \emph{integral commutative residuated lattice}. A semihoop $\textbf{A}$ is \emph{bounded} if $(A,\wedge,\vee, 1)$ has a least element $0$. An \emph{MTL-algebra} is a bounded semihoop, hence, MTL-algebras are prelinear integral bounded commutative residuated lattices, as usually defined \cite{EG2001,HM2009,NEG2005}. An MTL-algebra $\textbf{A}$ is an \emph{MTL chain} if its semihoop reduct is totally ordered. Let {\bf 1} and {\bf 2} be the only MTL-chains of one and two elements, respectively. For the rest of this paper we will refer to {\bf 1} as the \emph{trivial MTL-chain}.
\\

\noindent
It is known that the theory of MTL-algebras is a variety so we can realize the presentation of an algebraic theory. We write $\MTL$ for the algebraic category of MTL-algebras.
\\

\noindent
Let $\textbf{I}=(I,\leq)$ be a totally ordered set and $\mathcal{F}=\{\textbf{A}_{i}\}_{i\in I}$ a family of semihoops. Let us assume that the members of $\mathcal{F}$ share (up to isomorphism) the same neutral element; i.e, for every $i\neq j$, $A_{i}\cap A_{j}=\{1\}$. The \emph{ordinal sum} of the family $\mathcal{F}$, is the structure $\bigoplus_{i\in I} A_{i}$ whose universe is $\bigcup_{i\in I} A_{i}$ and whose operations are defined as:
\begin{displaymath}
 x\cdot y= \left\{ \begin{array}{lcl}
             x\cdot_{i}y, & if & x,y\in A_{i}
             \\ y, & if & x\in A_{i},\; and\; y\in A_{j}-\{1\},\; with\; i>j,
             \\ x, & if & x\in A_{i}-\{1\},\; and\; y\in A_{j},\; with\; i<j.
             \end{array}
   \right.
\end{displaymath}

\begin{displaymath}
 x\rightarrow y= \left\{ \begin{array}{lcl}
             x\rightarrow_{i}y, & if & x,y\in A_{i}
             \\ y, & if & x\in A_{i},\; and\; y\in A_{j},\; with\; i>j,
             \\ 1, & if & x\in A_{i}-\{1\},\; and\; y\in A_{j},\; with\; i<j.
             \end{array}
   \right.
\end{displaymath}
\noindent
where the subindex $i$ denotes the application of operations in $A_{i}$.
\\

\noindent
Moreover, if $\textbf{I}$ has a minimum $\perp$, $A_{i}$ is a totally ordered semihoop for every $i\in \textbf{I}$ and $A_{\perp}$ is bounded then $\bigoplus_{i\in I} A_{i}$ becomes a MTL-chain.
\\

\noindent
Let $M$ be an MTL-algebra. A submultiplicative monoid $F$ of $M$ is called a filter if is an up set respect to the order of $M$. In particular, for every $x\in F$, we write $\langle x \rangle$ for the filter generated by $x$; i.e.,
\[\langle x \rangle = \{a \in F \ | \ x^n \leq a \; \text{for some } n \in \mathbb{N}\}.\]
\noindent
For any filter $F$ of $M$, we can define a binary relation $\sim$, on $M$ by $a\sim b$ if and only if $a\rightarrow b\in F$ and $b\rightarrow a\in F$. A straightforward verification shows that $\sim$ is a congruence on $M$. For every $a\in M$, we write $[a]$ for the equivalence class of $a$ in $M/F$. Recall that (Section 3 of \cite{CMZ2016}) the canonical homomorphism $h:A\rightarrow A/M$ has the \emph{universal property} of forcing all the elements of $M$ to be $1$; i.e, for every MTL-algebra $B$ and every MTL-morphism $f:A\rightarrow B$ such that $f(a)=1$ for every $a\in M$, there exists a unique MTL-morphism $g:A/M \rightarrow B$ making the diagram below
\begin{displaymath}
\xymatrix{
A \ar[r]^-{h} \ar[dr]_-{f} & A/M \ar@{-->}[d]^-{g}
\\
 & B
}
\end{displaymath}

\noindent
commute. A filter $F$ of $M$ is \emph{prime} if $0\notin F$ and $x\vee y \in F$ entails $x\in F$ or $y\in F$, for every $x,y\in M$. The set of prime filters of an MTL-algebra $M$ ordered by inclusion will called \emph{spectrum} and will be noted as $\mathrm{Spec}(M)$.

\section{Finite archimedean MTL-chains}
\label{section Finite archimedean MTL-chains}

In this section we bring a characterization for the archimidean finite MTL-chains in terms of their nontrivial idempotent elements. In addition we prove that every morphism of finite archimedean MTL-chains is injective.
\\

\noindent
A totally ordered MTL-algebra is said to be \emph{archimedean} if for every $x \leq y < 1$, there exists $n \in \mathbb{N}$ such that $y^n \leq x$.

\begin{lem}
\label{descending chain condition for finite archimedean MTL-chains}
Let $M$ be an MTL-chain. If there is an $a \in M$ such that for every $n \in \mathbb{N}$, $a^{n + 1} < a^n$, then $M$ is infinite.
\end{lem}

\begin{prop}
\label{finite archimedean MTL-chains}
A finite MTL-chain $M$ is archimedean if and only if $M={\bf 2}$ or $M$ does not have nontrivial idempotent elements.
\end{prop}
\begin{proof}
If $M={\bf 2}$ the proof is trivial. If $M\neq {\bf 2}$ and do not have nontrivial idempotents, there exists $a \neq 0, 1$ in $M$. Since $M$ is finite, by Lemma \ref{descending chain condition for finite archimedean MTL-chains}, there exists $n \in \mathbb{N}$ such that $a^{n + 1} = a^n$. If $a^n > 0$, $(a^n)^2 = a^n$, and hence, $M$ has a nontrivial idempotent, in contradiction with the fact that $M$ does not have nontrivial idempotents. Hence, there exists $n \in \mathbb{N}$ such that $a^n = 0$. Now, for $b < a$ in $M$, we have that $a^n \leq b$, from where we can conclude that $M$ is archimedean. On the other hand, let us assume that $M$ is archimedean but there exists an idempotent element $a \neq 0, 1$. Hence, $a^n = a$ for every $n \in \mathbb{N}$. If $b < a$ (for example, if $b = 0$), we have that for every $n \in \mathbb{N}$, $b < a \leq a^n$, contradicting the archimedeanity of $M$. In consequence, no such idempotent can exist.
\end{proof}

\begin{coro}\label{Equivalent forms for arquimedeanity}
For any finite nontrivial MTL-chain $M$, there are equivalent:
\begin{enumerate}
	\item[i.] $M$ is archimedean,
	\item[ii.] $M$ is simple, and
	\item[iii.] $M$ does not have nontrivial idempotent elements.
\end{enumerate}
\end{coro}
\noindent
In \cite{HM2009} Hor\v{c}ik and Montagna gave an equational characterization for the archimedean finite MTL-chains.

\begin{lem}[\cite{HM2009}, Lemma 6.6]
	\label{equation for arquinedeanity}
Let $M$ be a finite MTL-chain. Then, $M$ is archimedean if and only if for every $a, b \in M$,
\[
((a \to b) \to b)^2 \leq a \vee b.
\]
\end{lem}




\noindent
The last part of this section is devoted to obtain a description of the morphisms between finite arquimedean MTL-chains. Let $f:A\rightarrow B$ be a morphism of finite MTL-chains. As usual, we write $K_{f}$ for the kernel of $f$; i.e.,
\[K_{f}=\{x\in A\mid f(x)=1\}\]

\begin{lem}\label{Characterization of Monomorphisms}
Let $f:A\rightarrow B$ be a morphism of MTL-algebras. Then $f$ is injective if and only if $f(x)=1$ implies $x=1$.
\end{lem}
\begin{proof}
Let $f$ be an injective morphism of MTL-algebras, then $f(x)=1=f(1)$ implies $x=1$. On the other hand, let us assume that $f(x)=1$ implies $x=1$. If $f(a)=f(b)$ then $f(a)\leq f(b)$ and $f(b)\leq f(a)$, thus by general properties of the residual it follows that $f(a)\rightarrow f(b)=1$ and $f(b)\rightarrow f(a)=1$ so $f(a\rightarrow b)=1$ and $f(b\rightarrow a)=1$. From the assumption we get that $a\rightarrow b=1$ and $b\rightarrow a=1$, thus $a\leq b$ and $b\leq a$. Hence, $a=b$ so $f$ is injective.
\end{proof}

\begin{lem}\label{Arquimedeanity restricts to monos}
Let $f:A\rightarrow B$ be a morphism of finite MTL-chains. If $A$ is archimedean then $B={\bf 1}$ or $f$ is injective.
\end{lem}
\begin{proof}
Since $A$ is archimedean, by $(ii)$ of Corollary \ref{Equivalent forms for arquimedeanity} we get that if is also simple so $K_{f}=A$ or $K_{f}=\{1\}$. In the first case, we get that $f(a)=1$ for every $a\in A$, so in particular $f(0)=0=1$, hence $B={\bf 1}$. In the last case, it follows that $f(a)=1$ implies $a=1$, so by Lemma \ref{Characterization of Monomorphisms} $f$ is injective.
\end{proof}

\begin{coro}\label{With archimedean monos reflects zero}
Let $f:A\rightarrow B$ be a morphism of finite MTL-chains. If $A$ is archimedean and $B\neq {\bf 1}$ then $f(x)=0$ implies $x=0$.
\end{coro}
\begin{proof}
Suppose that there exists $a\in A$ such that $f(a)=0$ but $a\neq 0$. Since $A$ is arquimedean, by Lemma \ref{Arquimedeanity restricts to monos}, we get that $f$ is injective so, from \[f(a^2)= f(a)^{2}=0=f(a),\] we obtain that $a=a^{2}$. On the other hand, since $f(a)\neq 1$, again from Lemma \ref{Arquimedeanity restricts to monos}, we get that $a\neq 1$ and consequently $0<a<1$. So $A$ possesses a non trivial idempotent which, by $(iii)$ of Corollary \ref{Equivalent forms for arquimedeanity}, is absurd.
\end{proof}

\begin{rem}\label{morphisms between archimedean chains}
Observe that every morphism of MTL-algebras between finite archimedean MTL-chains is injective. Let $f:A\rightarrow B$ be a morphism of finite arquimedean MTL-chains. Since $A$ is archimedean, from Lemma \ref{Arquimedeanity restricts to monos}, it follows that $f$ is injective or $B={\bf 1}$. Since $B$ is archimedean, by assumption, it follows, from Proposition \ref{finite archimedean MTL-chains}, that $B$ cannot be trivial. Hence, $f$ must be injective.
\end{rem}

\section{Finite labeled forests}
\label{section Finite labeled forests}

It is a very known fact that if $M$ is an BL-algebra, then, its dual spectrum is a forest (c.f.  Proposition 6 of \cite{TU1999} ). Such relation has been used to establish functorial correspondences before between BL-algebras and certain kind of labeled forests\footnote{Actually in \cite{ABM2009} the authors use the name \emph{weighted} instead of labeled.} (c.f. \cite{ABM2009}). Motivated by these ideas, in this section we show that there exist a functor from the category of finite MTL-algebras to the category of finite labeled forests. To do so, we will take advantage of the intimate relation between idempotent elements and filters, that is given for the case of finite MTL-algebras. This particular condition allows us to describe the spectrum of a finite MTL-algebra in terms of its join irreducible idempotent elements, as well as charaterize the quotients that result arquimedean MTL-chains.
\\

\noindent
A \emph{forest} is a poset $X$ such that for every $a\in X$ the set
\[\downarrow a=\{ x\in X \mid x\leq a \} \]
is a totally ordered subset of $X$.
\\

\noindent
This definition is motivated by the following result, whose proof is similar to the dual of Proposition 6 of \cite{TU1999}.

\begin{lem}
	\label{Spec is a root system}
	Let $M$ be a (finite) MTL-algebra. Then $\mathrm{Spec}(M)^{op}$ is a (finite) forest.
\end{lem}

\noindent
A \emph{tree} is a forest with a least element. A \emph{p-morphism} is a morphism of posets $f:X\rightarrow Y$ satisfying the following property: given $x\in X$ and $y\in Y$ such that $y\leq f(x)$ there exists $z\in X$ such that $z\leq x$ and $f(z)=y$. Let $\fMTL$ be the algebraic category of finite MTL-algebras.  We write $\faMTL$ for the algebraic category of finite archimedean MTL-algebras and $\faMTLc$ for the full subcategory of finite archimedean MTL-chains. Let $\Skel$ be the skeleton of $\faMTLc$. A \emph{labeled forest} is a function $l:F \rightarrow \Skel$, such that $F$ is a forest and the collection of archimedean MTL-chains $\{l(i)\}_{i\in F}$ (up to isomorphism) shares the same neutral element $1$. Consider two labeled forests $l:F \rightarrow \Skel$ and $m:G \rightarrow \Skel$. A morphism $l\rightarrow m$ is a pair $(\varphi, \mathcal{F})$ such that $\varphi: F \rightarrow G$ is a p-morphism and $\mathcal{F}=\{f_{x}\}_{x\in F}$ is a family of injective morphisms $f_{x}: (m\circ \varphi) (x) \rightarrow l(x)$ of MTL-algebras.
\\

\noindent
Let $(\varphi, \mathcal{F}):l \rightarrow m$ and $(\psi, \mathcal{G}):m \rightarrow n$ be two morphism between labeled forests. We define the composition $(\varphi, \mathcal{F})(\psi, \mathcal{G}):l\rightarrow n$ as the pair $(\psi\varphi,\mathcal{M})$, where $\mathcal{M}$ is the family whose elements are the MTL-morphims $f_{x}g_{\varphi(x)}:n(\psi\varphi)(x)\rightarrow l(x)$ for every $x\in F$. We will call $\fLF$ the category of labeled forests and its morphisms. The details of checking that $\fLF$ is a category are left to the reader.
\\

\noindent
Let $M$ be an MTL-algebra. We write $\mathcal{I}(M)$ for the poset of idempotent elements of $M$; i.e., \[\mathcal{I}(M) := \{x \in M \ | \ x^2 = x \}.\]

\begin{lem}
	\label{Filterequivalent}
In any MTL-algebra $M$, there are equivalent,
\begin{enumerate}
	\item[i.] $a \in \mathcal{I}(M)$, and
	\item[ii.] $\langle a \rangle = \ \uparrow a$
\end{enumerate}
\end{lem}
\begin{proof}
 	Let us assume that $\langle a \rangle = \ \uparrow a$. Since $a^{2}\in \langle a \rangle$ then  $a^{2}\in \uparrow a$, so $a\leq a^{2}$. Finally, since $M$ is negatively ordered,  $a^{2}\leq a$. Therefore $a^{2}=a$. The last part of the proof follows directly from the definition.
\end{proof}

\begin{coro}\label{idempotentdeterminesfilters}
Let $M$ be a finite MTL-algebra and $F \subseteq M$ a filter in $M$. There exists a unique $a \in \mathcal{I}(M)$ such that $F = \uparrow a$.
\end{coro}
\begin{proof}
 Since $M$ is finite, every filter $F \subseteq M$ is principal, so by Lemma \ref{Filterequivalent}, $F=\uparrow a$ for some $a\in \mathcal{I}(M)$. If there exists $a'\in \mathcal{I}(M)$ such that $\uparrow a= \uparrow a'$, then $a\leq a'$ and $a'\leq a$.
\end{proof}
\noindent
Let $M$ be a finite MTL-algebra. From the Corollary \ref{idempotentdeterminesfilters}, it follows that there is a bijection between $\mathcal{I}(M)$ and the filters of $M$. Let $\nzid$ the subposet of join irreducible elements of $\mathcal{I}(M)$. A direct application of Birkhoff's duality brings the following result.

\begin{coro}\label{joinirreducibledeterminesprimefilters}
	Let $M$ be a finite MTL-algebra and $P \in \mathrm{Spec}(M)$. Then, there exists a unique $e \in \nzid$ such that $P = \uparrow e$.
\end{coro}

\begin{coro}
\label{Spec iso to JI}
	Let $M$ be a finite MTL-algebra. The posets $\mathrm{Spec}(M)^{op}$ and $\nzid$ are isomorphic.
\end{coro}
\begin{proof}
 Let $\varphi: \nzid \rightarrow Spec(M)$ be the mapping defined as $\varphi(e)=\uparrow e$. From Corollary \ref{joinirreducibledeterminesprimefilters}, it follows that $\varphi$ is bijective. The proofs of the antimonotonicity of $\varphi$ and $\varphi^{-1}$ are straightforward.
\end{proof}

\begin{lem}\label{Closurejoinirreducibles}
Let $M$ be a finite MTL-algebra and $x\in \mathcal{I}(M)$ such that $x\neq 0$. If there exists some $k\in \nzid$ such that $x\leq k$ then $x$ is join irreducible.
\end{lem}
\begin{proof}
Suppose that $0<x\leq k$ for some $k\in \nzid$, with $x\in \mathcal{I}(M)$, then $\uparrow k\subseteq \uparrow x$. Suppose that $x\leq a\vee b$ but $x\nleq a,b$ for some $a,b\in M$, then $a,b\notin \uparrow x$ implies that $a,b \notin \uparrow k$ so $k \nleq a\vee b$, since $k\in \nzid$. Thereby $x\nleq a\vee b$, which is absurd, by assumption.
\end{proof}

\begin{rem}\label{chains are characterized by joinirreducibles}
Let $M$ be a finite MTL-algebra. Observe that, from Lemmas $1.2.8$ and $6.1.2$ of \cite{Z2016} it follows that $M/\uparrow e$ is a MTL-chain if and only if $e\in \JI{M}$.
\end{rem}
\noindent
We write $m(M)$ for the minimal elements of $\nzid$.

\begin{lem}\label{uniqejoinirreducible}
Let $M$ be an MTL-algebra and $e\in \nzid$. Then, there exists a unique $k\in \nzid \cup \{0\}$ such that $k\prec e$, where $\prec$ denotes the covering relation in posets.
\end{lem}
\begin{proof}
Let $e\in \nzid$, then either $e\in m(M)$ or $e\notin m(M)$. In the first case, the result follows, since $0\prec e$. In the second case, by Lemma \ref{Spec is a root system} and Corollary \ref{joinirreducibledeterminesprimefilters} we get that $\downarrow e \cap \nzid$ is a finite chain. If we consider $k$ as the coatom of the latter chain, the result holds.
\end{proof}
\noindent
Let $e\in \nzid$. In the following, we will write $a_{e}$ to denote the join irreducible element associated to $e$ in Lemma \ref{uniqejoinirreducible}. Note that $a_{e}=0$ if and only if $e\in m(M)$.

\begin{lem}\label{minimalcharactirizesarchimedean}
Let $M$ be a finite MTL-algebra and $e\in \nzid$. Then $M/ \uparrow e$ is archimedean if and only if $e\in m(M)$.
\end{lem}

\begin{proof}
Suppose $e\in m(M)$. From Remark \ref{chains are characterized by joinirreducibles}, we get that $M/\uparrow e$ is an MTL-chain. Now, if $M/\uparrow e$ is not archimedean, by Proposition \ref{finite archimedean MTL-chains}, there exists some $[k]\in \mathcal{I}(M/\uparrow e)$ such that $[k]^{2}=[k^{2}]=[k]$ with $[k]$ different than $[0]$ and $[1]$. So, since $ek^{2}\leq k$, $ek\leq ek^{2}$ and $e\in \mathcal{I}(M)$ we get that $(ek)^{2}=ek^{2}=ek$. Thus $ek\in \mathcal{I}(M)$. Observe that due to $[k]\neq 0$ we obtain that $ek\neq 0$. Since $ek\leq e$, from Lemma \ref{Closurejoinirreducibles} we get that $ek\in \nzid$. In consequence, $ek\leq e$, which is absurd because $e\in m(M)$.
\\
On the other hand, let assume that $M/\uparrow e$ is archimedean. If there is some $k\in \nzid$ such that $k\leq e$. Thus $[k]\in \mathcal{I}(M/\uparrow e)$. So by Lemma \ref{finite archimedean MTL-chains}, $[k]=[0]$ or $[k]=[1]=[e]$. If $[k]=[0]$ we get that $ek\leq 0$ then $k\leq e\rightarrow 0$. Since $k\in \mathcal{I}(M)$ and the product is monotone, $k\leq e(e\rightarrow 0)\leq 0$. Thereby $k=0$, which is absurd because $k\in \nzid$. In the case of $[k]=[e]$ we obtain that $e\leq k$. Since $k\leq e$ by assumption then we conclude that $e=k$. Therefore, $e\in m(M)$.

\end{proof}


\begin{rem}\label{filtersareMTL}
Let $M$ be a finite MTL-algebra and $F\subseteq M$ a filter. Let us check that $(F,\vee,\wedge,1,x)$ is an MTL-algebra such that $0_{F}=x$. By definition of $F$, $(F,\cdot,1)$ is a commutative monoid so, if $a,b\in F$ then $ab\in F$. The integrality of $M$ implies that $a\leq b\rightarrow a$ and $b\leq a\rightarrow b$, so since $F$ is an up set of $M$ then for every $a,b\in F$, we get that $a\rightarrow b,b\rightarrow a\in F$. Similarly, since $ab\leq a\wedge b$, by applying the last argument we get that $a\wedge b\in F$.  The proof for $a\vee b\in F$ is the same. Finally, due to Corollary \ref{idempotentdeterminesfilters} there exists a unique $x\in \mathcal{I}(M)$ such that $F=\uparrow x$, which is equivalent to say that $x=0_{F}$.
\end{rem}

\begin{lem}\label{labelingfMTL}
Let $M$ be a finite MTL-algebra, then $(\uparrow a_{e})/(\uparrow e)$ is an archimedean MTL chain. for every $e\in \nzid$.
\end{lem}
\begin{proof}
Recall that by Remark \ref{filtersareMTL} and Lemma \ref{uniqejoinirreducible}, we get that $\uparrow a_{e}$ is a finite MTL-algebra whose least element is $a_{e}$. Since $a_{e}\prec e$, it  follows that $\uparrow e$ is a proper filter of $\uparrow a_{e}$ with $e\in m(\uparrow a_{e})$. Therefore, from Lemma \ref{minimalcharactirizesarchimedean} we get that $\uparrow a_{e}/\uparrow e$ is an archimedean MTL chain.
\end{proof}

\noindent
Let $M$ and $N$ be finite MTL-algebras and $f: M \to N$ a morphism of MTL-algebras. It is a known fact (c.f. \cite{MUR2008}) that the assignments $M \mapsto Spec(M)$ and $f\mapsto Spec(f)=f^{-1}$, determines a contravariant functor $\textbf{Spec}:\fMTL \rightarrow \fCoh$ from the category of finite MTL-algebras into the category of finite Coherent (or Spectral) spaces.
\\

\noindent
Let $\varphi_{M}$ be the isomorphism between $\nzid$ and $Spec(M)^{op}$ of Proposition \ref{Spec iso to JI}.

\begin{lem}
	\label{assignation p morphism}
	Let $M$ and $N$ be finite MTL-algebras and $f: M \to N$ an MTL-algebra morphism. There exists a unique p-morphism $f^\ast: \JI{N} \to \JI{M}$ making the following diagram
	\begin{displaymath}
	\xymatrix{
	\JI{N} \ar[r]^-{f^\ast} \ar[d]_-{\varphi_{N}} & \JI{M} \ar[d]^-{\varphi_{M}}
	\\
	Spec(N) \ar[r]_-{spec(f)} & Spec(M)
	}
	\end{displaymath}
	commute.
\end{lem}
\begin{proof}
	
	Since $\varphi_{M}$ is an isomorphism, we get that $f^\ast=\varphi_{M}^{-1}spec(f)\varphi_{N}$. Observe that this map is defined as $f^\ast(e)=min\; S_{e}$ where $S_{e}= f^{-1}(\uparrow e)\cap \JI{M}$. In order to check the monotonocity, let $e\leq g$ in $\JI{N}$, then $\uparrow g\subseteq \uparrow e$, thus $f^{-1}(\uparrow g)\subseteq f^{-1}(\uparrow e)$ so $\uparrow f^{\ast}(g)\subseteq \uparrow f^{\ast}(e)$. Thereby, $f^{\ast}(e)\subseteq f^{\ast}(g)$. It only remains to check that $f^\ast$ is a p-morphism. To do so, let $g\in \JI{N}$ and $e\in \JI{M}$ such that $g\leq f^{\ast}(e)$. Since $\JI{N}$ is finite, we can consider $m=min\; S$, with \[S= \{k\in \JI{N}\mid k\leq e, \; g\leq f^{\ast}(k) \}. \]
	
We will prove that $f^{\ast}(m)=g$. Let $x\in \mathcal{I}(N)$ be such that $g\leq x$. Since $e\leq f(g)$, it follows that $f(x)\neq 0$. Consider $\uparrow mf(x)$. Since $mf(x)\leq m$, by Lemma \ref{Closurejoinirreducibles} we get that $mf(x)\in \JI{M}$. Let us very that $g\leq f^{\ast}(mf(x))$ by checking $f^{-1}(\uparrow mf(x))\subseteq \uparrow g$. If $b\in f^{-1}(\uparrow mf(x))$, then $mf(x)\leq f(b)$ and hence $m\leq f(x)\rightarrow f(b)=f(x\rightarrow b)$. Therefore $x\rightarrow b\in f^{-1}(\uparrow m)$. By construction of $m$, we have that $g\leq f^{\ast}(m)$, so $f^{-1}(\uparrow m)\subseteq \uparrow g$. Consequently, $g\leq x\rightarrow b$. Since $g\leq x$, we obtain that $g\leq x(x\rightarrow b)\leq b$. Hence, $mf(x)=m$, because $mf(x)\in S$. Finally, since $mf(x)=m\leq f(x)$, it follows that $x\in f^{-1}(\uparrow m)$. Thus $\uparrow x\subseteq f^{-1}(\uparrow m)$ and since $x\leq g$, we obtain that $\uparrow g\subseteq f^{-1}(\uparrow m)$.Then we conclude that $f^{\ast}(m)\leq g$. This concludes the proof.

\end{proof}
\noindent
Let $M$ be a finite MTL-algebra and consider the function \[\etiq{M}:\JI{M}\rightarrow \Skel\] defined as $\etiq{M} (e)=\uparrow a_{e}/ \uparrow e$. Since from Lemma \ref{Spec is a root system} we know that $\nzid$ is a finite forest, $\etiq{M}$ is a finite labeled forest.
\\

\noindent
Let $\textbf{F}$ be a finite forest and $X\subseteq F$. We write $Min(X)$ for the minimal elements of $X$.

\begin{lem}\label{pmorphismspreserveminimals}
Let $f:X\rightarrow Y$ be a p-morphism. If $x\in Min(X)$ then $f(x)\in Min(Y)$.
\end{lem}
\begin{proof}
Let us assume $x\in Min(X)$, and suppose that there exists $y\in Y$ such that $y< f(x)$. Since $f$ is a p-morphism, there exists $z\in X$, with $z\leq x$ such that $y=f(z)$. Since $y\neq f(x)$, then $z\neq x$, so $x\notin Min(X)$. This fact is absurd by assumption.
\end{proof}

\begin{lem}\label{labelingMTLmorphisms}
Let $M$ and $N$ be finite MTL-algebras and $f: M \to N$ an MTL-algebra morphism. Then, for every $e\in \JI{N}$, $f$ determines a morphism $\overline{f}_{e}:\uparrow a_{f^{\ast}(e)} \to \uparrow a_{e}$ such that there exists a unique MTL-algebra morphism $f_{e}:\uparrow a_{f^{\ast}(e)}/ \uparrow f^{\ast}(e) \to \uparrow a_{e}/ \uparrow e$ making the diagram
\begin{displaymath}
\xymatrix{
\uparrow a_{f^{\ast}(e)} \ar[r]^-{\overline{f}_{e}} \ar[d] & \uparrow a_{e} \ar[d]
\\
\uparrow a_{f^{\ast}(e)}/ \uparrow f^{\ast}(e) \ar[r]_-{f_{e}} & \uparrow a_{e}/ \uparrow e
}
\end{displaymath}
\noindent
commute.
\end{lem}
\begin{proof}
Let $e\in \JI{N}$. Then $e\notin m(N)$ or $e\in m(N)$. In the first case, it follows that $a_{e} > 0_{N}$ and thus, $\uparrow a_{e}\subset N$. Since $a_{e}\leq e$ and $f^{\ast}$ is monotone then $f^{\ast}(a_{e})\leq f^{\ast}(e)$. Since $a_{f^{\ast}(e)}\prec f^{\ast}(e)$, $f^{\ast}(a_{e})\leq a_{f^{\ast}(e)}$. Thereby, $\uparrow a_{f^{\ast}(e)}\subseteq  \uparrow f^{\ast}(a_{e})$. Let us define $\overline{f}_{e}:\uparrow a_{f^{\ast}(e)} \to \uparrow a_{e}$ as

\begin{displaymath}
\overline{f}_{e}(x)= \left\{ \begin{array}{ll}
             f(x), &  x> a_{f^{\ast}(e)}\\
             a_{e}, &  x= a_{f^{\ast}(e)}
             \end{array}
   \right.
\end{displaymath}
\noindent
From Lemma \ref{assignation p morphism}, we get that $\uparrow f^{\ast}(a_{e})=f^{-1}(\uparrow a_{e})$ and $\overline{f}_{e}$ is a well defined MTL-morphism. Let us consider $a_{f^{\ast}(e)}<f^{\ast}(e)\leq x$, then, $\overline{f}_{e}(a_{f^{\ast}(e)})<\overline{f}_{e}(f^{\ast}(e))\leq \overline{f}_{e}(x)$ since $\overline{f}_{e}$ is monotone. By definition of $\overline{f}_{e}$ we obtain that $a_{e}<f(f^{\ast}(e))\leq f(x)$. Then, applying Lemma \ref{assignation p morphism}, we get that $e\leq f(f^{\ast}(e))$, so we can conclude that $e\leq f(x)$. This means that $[\overline{f}_{e}(x)]=[1]$ in $\uparrow a_{e}/ \uparrow e$. Hence, by the universal property of quotients in $\MTL$, there exists a unique MTL-morphism $f_{e}:\uparrow a_{f^{\ast}(e)}/ \uparrow f^{\ast}(e) \to \uparrow a_{e}/ \uparrow e$ making the diagram above commutes.

\noindent
Finally, if $e\in m(N)$, we get that $a_{e}=0_{N}$ and $\uparrow a_{e}=N$. Since $f^{\ast}$ is a p-morphism, due to Lemma \ref{pmorphismspreserveminimals}, $f^{\ast}(e)\in m(M)$, $a_{f^{\ast}(e)}=0_{M}$ and consequently, $\uparrow a_{f^{\ast}(e)}=M$. Let $\overline{f}_{e}=f$. The proof of $[\overline{f}_{e}(x)]=[1]$ in $\uparrow a_{e}/ \uparrow e$ is similar to the given for first case. The rest of the proof follows from the universal property of quotients in $\MTL$.

\end{proof}
\noindent
Let $f: M\rightarrow N$ be an MTL-morphism between finite MTL-algebras and $\mathcal{F}_{f}:=\{f_e\}_{e\in \JI{N}}$ be the family of MTL-morphisms obtained in Lemma \ref{labelingMTLmorphisms}.

\begin{coro}
Let $M$ and $N$ be finite MTL-algebras and $f: M \to N$ an MTL-algebra morphism. Then the pair $(f^{\ast},\mathcal{F}_{f})$ is a morphism between the labeled forests $\etiq{N}$ and $\etiq{M}$.
\end{coro}

\begin{theo}
	\label{functor from fMTL to root systems}
 The assignments $M\mapsto l_{M}$ and $f\mapsto (f^{\ast},\mathcal{F}_{f})$ define a contravariant functor \[\Ida : \fMTL \rightarrow \fLF.\]
\end{theo}
\begin{proof}
	
	Let $f:M\rightarrow N$ and $g:N\rightarrow O$ be morphisms in $\fMTL$ and consider the diagram
\begin{displaymath}
\xymatrix{
Spec(O) \ar[r]^-{Spec(g)} & Spec(N) \ar[r]^-{Spec(f)} \ar@<1ex>[d]^-{{\varphi _{N}}^{-1}} & Spec(M)
\\
\JI{O} \ar[u]^-{{\varphi _{O}}} \ar[r]_-{g^{\ast}} & \JI{N} \ar@<1ex>[u]^-{{\varphi _{N}}} \ar[r]_-{f^{\ast}} & \JI{M} \ar[u]^-{{\varphi _{M}}}
}
\end{displaymath}
\noindent
associated to the composition $gf: M\rightarrow O$. Since $\textbf{Spec}$ is contravariant,
\begin{equation}
(gf)^{\ast}=\varphi_{M}^{-1}Spec(gf)\varphi_{O}=(\varphi_{M}^{-1}Spec(f)\varphi_{N})(\varphi_{N}^{-1}Spec(g)\varphi_{O})=f^{\ast}g^{\ast}.\label{eq1}
\end{equation}
	
\noindent
If $id_{M}$ denotes the identity map of the MTL-algebra $M$, a straightforward calculation proves that $(id_{M})^{\ast}=id_{\JI{M}}$. On the other hand, let $e\in \JI{O}$. We will verify that $\overline{(gf)}_{e}=\overline{g}_{e}\overline{f}_{g^{\ast}(e)}$. From (\ref{eq1}), we conclude that $a_{(gf)^{\ast}(e)}=a_{f^{\ast}(g^{\ast}(e))}$. If $x>a_{(gf)^{\ast}(e)}$, then $x>a_{f^{\ast}(g^{\ast}(e))}$. By the monotonicity of $f$ and Lemma \ref{assignation p morphism}, we obtain that $a_{g^{\ast}(e)}\leq f(a_{f^{\ast}(g^{\ast}(e))})\leq f(x)$. Hence $a_{g^{\ast}(e)}\leq f(x)$. In a similar way, by the monotonicity of $g$ and using again Lemma \ref{assignation p morphism}, we get that $a_{e}\leq (gf)(x)$. Therefore, for every $x\in \uparrow a_{(gf)^{\ast}(e)}$ it follows that $\overline{(gf)}_{e}(x)=\overline{g}_{e}\overline{f}_{g^{\ast}(e)}(x)$. Finally, from Lemma \ref{labelingMTLmorphisms}, we obtain that $g_{e}f_{g^{\ast}(e)}=(gf)_{e}$. So, for every $e\in \JI{N}$ the diagram below

\begin{displaymath}
\xymatrix{
\uparrow a_{(gf)^{\ast}(e)} \ar[r]^-{(\overline{f})_{g^{\ast}(e)}} \ar[d] \ar@/^2pc/[rr]^-{\overline{(gf)}_{e}} & \uparrow a_{g^{\ast}(e)} \ar[r]^-{\overline{g}_{e}} \ar[d] & \uparrow a_{e} \ar[d]
\\
\uparrow a_{(gf)^{\ast}(e)}/\uparrow (gf)^{\ast}(e) \ar[r]_{f_{g^{\ast}(e)}} \ar@/_2pc/[rr]_-{(gf)_{e}} & \uparrow a_{g^{\ast}(e)}/ \uparrow g^{\ast}(e) \ar[r]_-{g_{e}} & \uparrow a_{e}/ \uparrow e
}
\end{displaymath}
\noindent
commutes. Therefore $(g^{\ast}f^{\ast}, \mathcal{F}) =((gf)^{\ast},\mathcal{F}_{gf})$, where $\mathcal{F}_{gf}=\{(gf)_{e}\}_{e\in \JI{O}}$ and $\mathcal{F}=\{g_{e}f_{g^{\ast}(e)}\}_{e\in \JI{O}}$. Hence $\Ida(gf)=\Ida(f)\Ida(g)$. From the definition of $\Ida$ it easily follows that $\Ida(id_{M})=id_{\Ida(M)}$.
\end{proof}

\section{Forest Product of MTL-algebras}
\label{section Forest Product of MTL-algebras}

In this section we introduce the notion of forest product. It is simply a poset product as defined in \cite{BM2011} when restricted to posets which are forests. For the sake of completeness, we give explicitly the necessary definitions.

\begin{defi}\label{Poset Product}
Let $\textbf{F}=(F,\leq)$ be a forest and let $\{\textbf{M}_{i}\}_{i\in \textbf{F}}$ a collection of MTL-chains such that, up to isomorphism, all they share the same neutral element $1$. If $\left( \bigcup_{i\in \textbf{F}}\textbf{M}_{i}\right)^{F}$ denotes the set of functions $h:F\rightarrow \bigcup_{i\in \textbf{F}} \textbf{M}_{i}$ such that $h(i)\in \textbf{M}_{i}$ for all $i\in \textbf{F}$, the \emph{forest product} $\bigotimes_{i\in \textbf{F}}\textbf{M}_{i}$ is the algebra $\textbf{M}$ defined as follows:
\begin{itemize}
\item[(1)] The elements of $\textbf{M}$ are the $h\in \left(\bigcup_{i\in \textbf{F}}\textbf{M}_{i}\right)^{F}$ such that, for all $i\in \textbf{F}$ if $h(i)\neq 0_{i}$ then for all $j<i$, $h(j)=1$.
\item[(2)] The monoid operation and the lattice operations are defined pointwise.
\item[(3)] The residual is defined as follows:
\begin{displaymath}
(h\rightarrow g)(i)=\left\{ \begin{array}{lc}
             h(i)\rightarrow_{i} g(i), &  \text{if for all}\; j<i,\; h(j)\leq_{j} g(j)  \\
             \\ 0_{i}  & \text{otherwise} \\
             \end{array}
   \right.
\end{displaymath}
 where de subindex $i$ denotes the application of operations and of order in $\textbf{M}_{i}$.
\end{itemize}
\end{defi}

The following result is a slight modification of Theorem 3.5.3 in \cite{BM2011}.

\begin{lem}\label{forest product is a MTL-algebra}
The forest product of MTL-chains is an MTL-algebra.
\end{lem}

\noindent
In the following if we refer to a collection $\{\textbf{M}_{i}\}_{i\in \textbf{F}}$ of MTL-chains indexed by a forest $\textbf{F}$ we always will assume that it satisfies the conditions of Definition \ref{Poset Product}.

\begin{lem}\label{Equivalent forms of forest product}
Let $\textbf{F}$ be a forest and $\{\textbf{M}_{i}\}_{i\in \textbf{F}}$ a collection of MTL-chains. There are equivalent:
\begin{itemize}
\item[1.] $h\in \bigotimes_{i\in \textbf{F}}\textbf{M}_{i}$,
\item[2.] For every $i<j$ in $\textbf{F}$, $h(j)=0_{j}$ or $h(i)=1$,
\item[3.] For all $i\in \textbf{F}$ if $h(i)\neq 1$ then for all $i<j$, $h(j)=0_{j}$,
\item[4.] $\bigcup_{i\in \textbf{F}} h^{-1}(0_{j})$ is an upset of $\textbf{F}$, $h^{-1}(1)$ is a downset of $\textbf{F}$ and \[C_{h}=\{i\in \textbf{F} \mid h(i)\notin \{0_{i},1\}\},\] is a (possibly empty) antichain of $\textbf{F}$.
\end{itemize}
\end{lem}
\begin{proof}
Since the implications $(1)\Rightarrow (2) \Rightarrow (3)$ follow straight from definition, we only prove the remaining implications. Let us start proving that $(3)$ implies $(4)$: To prove that $h^{-1}(1)$ is a downset of $\textbf{F}$ we proceed by contradiction. Suppose $i<j$ with $h(j)=1$ but $h(i)\neq 1$. Thus, $h(j)=0_{i}$, by assumption, which is absurd. To prove that $\bigcup_{i\in I} h^{-1}(0_{j})$ is an upset of $\textbf{F}$ let us suppose that $i<j$ with $h(i)=0_{i}$, thus, since $h(i)\neq 1$, and $(3)$, we get that $h(j)=0_{j}$. If $C_{h}$ is not an antichain, there exist $i,j\in C_{h}$ comparable. Without loss of generality, we can assume $i<j$, $h(i)\neq 1$ and $h(j)\neq 0_{j}$, then because of $(3)$, we obtain that $h(j)= 0_{j}$, which is absurd. Finally, to prove that $(4)$ implies $(1)$, let $h\in \left( \bigcup_{i\in \textbf{F}}\textbf{M}_{i}\right)^{F}$ and suppose that $i<j$ with $h(i)\neq 1$. If $h(j)\neq 0_{j}$, thus $i,j\in C_{h}$, which is absurd, since $C_{h}$ is by assumption an antichain.
\end{proof}

\begin{rem}\label{elements of chain products}
Let $\textbf{F}$ be a chain and $\{\textbf{M}_{i}\}_{i\in \textbf{F}}$ a collection of MTL-chains. Let us consider $h\in \bigotimes_{i\in \textbf{F}}\textbf{M}_{i}$ and $j\in \textbf{F}$. Note that there are only two possible cases for $h(j)$, namely $h(j)\neq 0_{j}$ or $h(j)\neq 1$. If $h(j)\neq 1$, from $(4)$ of Lemma \ref{Equivalent forms of forest product}, it follows that for every $j<i$, $h(i)=0_{i}$ and due to $(3)$ of the same Lemma, $h(k)=1$ for every $k<j$. If $h(j)\neq 0_{j}$, from Definition \ref{Poset Product}, we get that $h(k)=0_{k}$ for every $j<k$ and by $(4)$ of Lemma \ref{Equivalent forms of forest product}, we have that $h(i)=1$ for every $i<j$.
\end{rem}
\noindent
A commutative integral residuated lattice $A$ is called \emph{really local} (Definition 1.2.27 of \cite{Z2016}) if it is not trivial $(0\neq 1)$ and for every $x,y\in A$,
\[ x\vee y=1 \; \Rightarrow \; x=1\; \text{or}\; y=1. \]

\begin{lem}\label{really local and prelinear are chains}
Let $A$ be a non-trivial commutative integral residuated lattice. If the canonical order of A is total then A is really local. If A is also prelinear then the converse holds.
\end{lem}
\begin{proof}
A direct consequence of Lemma 13.7 of \cite{CMZ2016}.
\end{proof}

\begin{lem}\label{chain products are MTL-chains}
Let $\textbf{F}$ be a forest and $\{\textbf{M}_{i}\}_{i\in \textbf{F}}$ a collection of MTL-chains. Then $\textbf{F}$ is a totally ordered set if and only if $\bigotimes_{i\in \textbf{F}}\textbf{M}_{i}$ is an MTL-chain.
\end{lem}
\begin{proof}
Suppose $\textbf{F}$ is a totally ordered set and let $g,h\in \bigotimes_{i\in \textbf{F}}\textbf{M}_{i}$ be such that $(g\vee h)=1$. Since the lattice operations in $\bigotimes_{i\in \textbf{F}}\textbf{M}_{i}$ are calculated pointwise, for every $i\in \textbf{F}$, $g(i)\vee h(i)=1$. If $g,h\neq 1$, there exists some $j\in \textbf{F}$ such that $g(j),h(j)\neq 1$. From Remark \ref{elements of chain products} it follows that $g(k)=h(k)=0_{k}$, for every $j<k$ so we get that $(g\vee h)(k)=g(k)\vee h(k)=0_{k}$, which contradicts our assumption. Hence, since every MTL-algebra is prelinear, from Lemma \ref{really local and prelinear are chains} we get that $\bigotimes_{i\in \textbf{F}}\textbf{M}_{i}$ is a MTL-chain. On the other hand, let us assume that $\bigotimes_{i\in \textbf{F}}\textbf{M}_{i}$ is a MTL-chain. If $\textbf{F}$ is not a totally ordered set, thus there exist two different elements $n$ and $m$ in $F$ which are not comparable. Let us consider $g,h\in \bigotimes_{i\in \textbf{F}}\textbf{M}_{i}$, defined as
\begin{center}
\begin{tabular}{ccc}
$g(i)=\left\{ \begin{array}{lc}
             0_{n}, &  \text{if}\; i\geq n  \\
             \\ 1,  &  \text{otherwise} \\
             \end{array}
   \right.$ &  & $h(i)=\left\{ \begin{array}{lc}
             0_{m}, &  \text{if}\; i\geq m  \\
             \\ 1,  & \text{otherwise} \\
             \end{array}
   \right.$
\end{tabular}
\end{center}

\noindent
Observe that $g\vee h=1$ but $g,h\neq 1$, which is a contradiction since, by Lemma \ref{really local and prelinear are chains}, $\bigotimes_{i\in \textbf{F}}\textbf{M}_{i}$ is really local.

\end{proof}

\noindent
Let $\textbf{F}$ be a forest and $\{\textbf{M}_{i}\}_{i\in \textbf{F}}$ a collection of MTL-chains. We write $\dec{\textbf{F}}$ for the collection of downsets of $\textbf{F}$. Let $S$ be a proper downset of $\textbf{F}$ and consider
\[\filter{S}:= \{h\in \bigotimes_{i\in \textbf{F}}\textbf{M}_{i} \mid h|_{S}=1 \} \]

\noindent
Observe that $\filter{S}$ is a proper filter of $\bigotimes_{i\in \textbf{F}}\textbf{M}_{i}$. Since $S^{c}$ is itself a forest, due to Lemma \ref{forest product is a MTL-algebra}, $\bigotimes_{i\in \textbf{S}^{c}}\textbf{M}_{i}$ is an MTL-algebra. Using the fact that every filter of a MTL-algebra is a semihoop, we obtain the following result.

\begin{lem}\label{Downsets give filters}
Let $\textbf{F}$ be a forest and $\{\textbf{M}_{i}\}_{i\in \textbf{F}}$ a collection of MTL-chains and $S\in \dec{\textbf{F}}$. Then $\filter{S}$ and $\bigotimes_{i\in \textbf{S}^{c}}\textbf{M}_{i}$ are isomorphic semihoops.
\end{lem}
\begin{proof}
Let $g\in \bigotimes_{i\in \textbf{S}^{c}}\textbf{M}_{i}$. Define $\varphi: \bigotimes_{i\in \textbf{S}^{c}}\textbf{M}_{i} \rightarrow \filter{S}$ as
\begin{displaymath}
\varphi(g)(i)=\left\{ \begin{array}{lc}
             g(i), &  \text{if}\; i\notin S  \\
             \\ 1,  & \text{if}\; i\in S \\
             \end{array}
   \right.
\end{displaymath}
For this part of the proof we will write $h$ to denote $\varphi(g)$. First, we prove that $h$ is well defined. Let us take $i<j$ in $F$ and suppose that $h(i)\neq 1$. By construction of $h$, we get that $i\notin S$, so $h(i)=g(i)$. If $h(j)\neq 0_{j}$, then $h(j)=1$ or $0_{j}<h(j)<1$. In the first case we obtain that $j\in S$ and since $i<j$ and $S\in \dec{\textbf{F}}$, $i\in S$, which is absurd. In the second case, since $S^{c}$ is an upset of $\textbf{F}$, from $i\notin S$ and $i<j$ it follows that $j\notin S$. Hence, $h(j)=g(j)\neq 1$. Therefore, there are $i,j\in C_{h}$ comparable, which by $(4)$ of Lemma \ref{Equivalent forms of forest product} is absurd. Consequently, $h(j)=0_{j}$ and thus, by $(3)$ of Lemma \ref{Equivalent forms of forest product}, we get that $h\in  \bigotimes_{i\in \textbf{F}}\textbf{M}_{i}$. By construction, it is clear that $h\in \filter{S}$. In order to verify that $\varphi$ is surjective, let $f\in \filter{S}$ and consider $f|_{S^{c}}$. Since $S^{c}$ is an upset and $f\in \bigotimes_{i\in \textbf{F}}\textbf{M}_{i}$, it is clear that $\varphi(f|_{S^{c}})=f$. The injectivity of $\varphi$ is immediate.
\\

\noindent
Since the monoid and lattice operations in $\filter{S}$ and $\bigotimes_{i\in \textbf{S}^{c}}\textbf{M}_{i}$ are defined pointwise it is clear that $\varphi$ preserve such operations. We prove that $\varphi$ preserve the residual. To do so, let $s,t\in \bigotimes_{i\in \textbf{S}^{c}}\textbf{M}_{i}$. Then,

\begin{equation}
(\varphi(s)\rightarrow \varphi(t))(i)= \left\{ \begin{array}{lc}
             \varphi(s)(i)\rightarrow_{i} \varphi(t)(i), &  \text{if for all}\; j<i, \; \varphi(s)(j)\leq_{j} \varphi(t)(j)  \\
             \\ 0_{i},  & \text{otherwise}\\
             \end{array} \label{aux1}
   \right.
\end{equation}
and
\begin{equation}
\varphi (s\rightarrow t)(i)= \left\{ \begin{array}{lc}
             (s\rightarrow t)(i), &  \text{if}\; i\notin S  \\
             \\ 1,  & \text{if}\; i\in S \\
             \end{array}
   \right. \label{aux2}
\end{equation}
\noindent
If $i\in S$ then $\varphi(s)=\varphi(t)=1$ so $\varphi(s)(i)\rightarrow_{i} \varphi(t)(i)=1$. Since $S\in \dec{\textbf{F}}$, if $j<i$, $j\in S$. Therefore $\varphi(s)(j)=\varphi(t)(j)=1$. Hence, for every $i\in S$, $\varphi (s\rightarrow t)(i)=(\varphi(s)\rightarrow \varphi(t))(i)$. On the other hand, if $i\notin S$ and $(\varphi(s)\rightarrow \varphi(t))(i)\neq 1$ then $(\varphi(s)\rightarrow \varphi(t))(i)=0_{i}$ or $0_{i}<(\varphi(s)\rightarrow \varphi(t))(i)<1$. In the first case, from the equation (\ref{aux1}) it follows that there exists $j\in F$ with $j<i$ such that $\varphi(s)(j)\nleq_{j} \varphi(t)(j)$. If $j\in S$, $\varphi(s)(j)=\varphi(t)(j)=1$. Therefore $j\notin S$. Hence, since $\varphi(s)(j)=s(j)$ and $\varphi(t)(j)=t(j)$ we get that there exists $j\notin S$ with $j<i$ such that $s(j)\nleq_{j} t(j)$. Then $(s\rightarrow t)(i)=0_{i}=(\varphi(s)\rightarrow \varphi(t))(i)$. In the second case, from the equation (\ref{aux2}) we get that $(\varphi(s)\rightarrow \varphi(t))(i)=\varphi(s)(i) \rightarrow_{i} \varphi(t)(i)$ so, by the definition of $\varphi$, we obtain that $\varphi(s)(i)\rightarrow \varphi(t)(i)=(\varphi(s)\rightarrow \varphi(t))(i)$. Finally, in the case $i\notin S$ and $(\varphi(s)\rightarrow \varphi(t))(i)=1$, $\varphi(s)(i)\leq_{i} \varphi(t)(i)$ and in consequence $\varphi(s)(i) \rightarrow_{i} \varphi(t)(i)=1$. Thus, for every $j<i$, $\varphi(s)(j)\leq_{j} \varphi(t)(j)$. In particular, if $j\notin S$ and $j<i$, $s(j)\leq_{j} t(j)$. Therefore $\varphi(s)(i)\rightarrow \varphi(t)(i)=(s\rightarrow t)(i)$. Hence, for every $i\notin S$, $(\varphi(s)\rightarrow \varphi(t))(i)=\varphi(s)(i) \rightarrow_{i} \varphi(t)(i)$. This concludes the proof.

\end{proof}

\begin{coro}\label{Description of particular filters}
Let $\textbf{F}$ be a forest, $S,T\in \dec{\textbf{F}}$ such that $S\subseteq T$ and $\{\textbf{M}_{i}\}_{i\in \textbf{F}}$ a collection of MTL-chains. Take
$ \Filter{S}{T}:=\{h\in \bigotimes_{i\in \textbf{T}}\textbf{M}_{i}\mid h|_{S}=1 \}$.
Then, $\Filter{S}{T}$ and $\bigotimes_{i\in \textbf{S}^{c}\cap T}\textbf{M}_{i}$ are isomorphic semihoops.
\end{coro}
\begin{proof}
Since $S\subseteq T$ and $S,T\in \dec{\textbf{F}}$ we get that $S\in \dec{\textbf{T}}$. The result follows from Lemma \ref{Downsets give filters}.
\end{proof}

\subsection{Forest products are sheaves}

In every poset $\textbf{F}$ the collection $\dec{\textbf{F}}$ of downsets of $\textbf{F}$ defines a topology over $F$ called the \emph{Alexandrov topology} on $\textbf{F}$. Let $S,T\in \dec{\textbf{F}}$ be such that $S\subseteq T$ and $\{\textbf{M}_{i}\}_{i\in \textbf{F}}$ be a collection of MTL-chains. Observe that if $h\in \bigotimes_{i\in \textbf{T}}\textbf{M}_{i}$ then the restriction $h|_{S}$ is an element of $\bigotimes_{i\in \textbf{S}}\textbf{M}_{i}$, so the assigment that sends $T\in \dec{\textbf{F}}$ to $\bigotimes_{i\in \textbf{T}}\textbf{M}_{i}$ defines a presheaf $\Pprod:\dec{\textbf{F}}^{op}\rightarrow \MTL$.

\begin{lem}\label{Description of actions of forest products}
Let $\textbf{F}$ be a forest and $\{\textbf{M}_{i}\}_{i\in \textbf{F}}$ a collection of MTL-chains. Then, for every $S\in \dec{\textbf{F}}$  \[\Pprod(S)\cong \Pprod(F)/\filter{S}. \]
\end{lem}
\begin{proof}
Let $r:\Pprod(F) \rightarrow \Pprod(S)$ be the restriction to $S$. It is clear that $r$ is a surjective morphism of MTL-algebras such that $r(h)=1$, for every $h\in \filter{S}$. Then, by the universal property of the canonical homomorphism $\beta:\Pprod(F) \rightarrow \Pprod(F)/\filter{S}$, there exists a unique morphism  of MTL-algebras $\alpha: \Pprod(F)/\filter{S}\rightarrow \Pprod(S)$ such that the diagram below
\begin{displaymath}
\xymatrix{
\Pprod(F) \ar[r]^-{\beta} \ar[dr]_-{r} & \Pprod(F)/\filter{S} \ar@{-->}[d]^-{\alpha}
\\
 & \Pprod(S)
}
\end{displaymath}
\noindent
commutes. Observe that $\alpha \beta = r$, so since $\beta$ is surjective, it follows that $\alpha$ is surjective too. The verification of the injectivity of $\alpha$ is straightforward.

\end{proof}

\begin{coro}
Let $\textbf{F}$ be a forest, $S,T\in \dec{\textbf{F}}$ such that $S\subseteq T$ and $\{\textbf{M}_{i}\}_{i\in \textbf{F}}$ a collection of MTL-chains. Then $\Pprod(S)\cong \Pprod(T)/\Filter{S}{T}$.
\end{coro}
\begin{proof}
Due to Lemma \ref{Description of actions of forest products}, $\Pprod(T)\cong \Pprod(F)/\filter{T}$. Observe that $\Filter{S}{T}\cong \filter{S}/\filter{T}$, thus the result follows as a direct consequence of the second isomorphism theorem (Theorem 6.15 of \cite{SB1981}).
\end{proof}

\begin{lem}\label{Chains characterized by prime filters}
Let $A$ be a non-trivial MTL-algebra. Then $A/P$ is an MTL-chain if and only if $P$ is a non trivial prime filter.
\end{lem}

\begin{lem}\label{Chains determines prime filters of poset products}
Let $\textbf{F}$ be a forest, $S\in \dec{\textbf{F}}$ and $\{\textbf{M}_{i}\}_{i\in \textbf{F}}$ a collection of MTL-chains. Then $\filter{S}$ is prime if and only if $S$ is totally ordered.
\end{lem}
\begin{proof}
Observe that asking $S$ to be totally ordered, is equivalent, by Lemma  \ref{chain products are MTL-chains}, to asking $\bigotimes_{i\in \textbf{S}}\textbf{M}_{i}$ to be an MTL-chain. By Lemma \ref{Description of actions of forest products}, $\bigotimes_{i\in \textbf{S}}\textbf{M}_{i}\cong \bigotimes_{i\in \textbf{F}}\textbf{M}_{i}/\filter{S}$. Hence, the result follows from the first remark and Lemma \ref{Chains characterized by prime filters}.
\end{proof}

\noindent
Let $\Shv(\textbf{P})$ be the category of sheaves over the Alexandrov space $(P,\dec{\textbf{P}})$. Since the theory of MTL-algebras is algebraic, it is well-known that an MTL-algebra in $\Shv(\textbf{P})$ is a functor $\dec{\textbf{P}}^{op} \rightarrow \MTL$ such that the composite presheaf $\dec{\textbf{P}}^{op} \rightarrow \MTL \rightarrow \Set$ is a sheaf.

\begin{lem}\label{Forest Product are sheaves}
Let $\textbf{F}$ be a forest and $\{\textbf{M}_{i}\}_{i\in \textbf{F}}$ a collection of MTL-chains. The presheaf $\Pprod:\dec{\textbf{P}}^{op}\rightarrow \MTL$, with $\Pprod(T)=\bigotimes_{i\in \textbf{T}}\textbf{M}_{i}$, is an MTL-algebra in $\Shv(\textbf{P})$.
\end{lem}
\begin{proof}
Suppose that $T=\bigcup_{\alpha \in I} S_{\alpha}$, with $S_{\alpha},T\in \dec{\textbf{F}}$, for every $\alpha \in I$, and let $h_{\alpha}\in \Pprod(S_{\alpha})$ be a matching family. Thus, for every $\alpha \neq \beta$ in $I$:

\begin{equation}\label{ecuation}
h_{\alpha}|_{S_{\alpha}\cap S_{\beta}}=h_{\beta}|_{S_{\alpha}\cap S_{\beta}}
\end{equation}

\noindent
Let us consider the following function:

\begin{displaymath}
\begin{tabular}{ccll}
$h:$ & $I$ & $\rightarrow$ & $\bigcup_{i\in \textbf{T}}\textbf{M}_{i}$
\\
 & $i$ & $\mapsto$ & $h_{\alpha}(i),$ if $i\in S_{\alpha}$
\end{tabular}
\end{displaymath}
\\

\noindent
Observe that (\ref{ecuation}) guarantees that $h$ is well defined. To check that $h\in \bigotimes_{i\in \textbf{T}}\textbf{M}_{i}$, let us suppose that $i\in T$ and $h(i)\neq 0_{i}$. If $j<i$ then, since $T=\bigcup_{\alpha \in I} S_{\alpha}$, there exists some $\beta \in I$ such that $i\in S_{\beta}$. In such case, $j\in S_{\beta}$, since $S_{\beta}\in \dec{\textbf{F}}$. Then $h(i)=h_{\beta}(i)$ and $h(j)=h_{\beta}(j)$. Since $h_{\beta}\in \bigotimes_{i\in \textbf{S}_{\beta}}\textbf{M}_{i}$, we conclude that $h(j)=h_{\beta}(j)=1$. Therefore $h$ amalgamates $\{h_{\alpha}\}_{\alpha \in I}$. To verify the uniqueness of $h$, let us suppose that there exists $f\in \Pprod(T)$ such that $f|_{S_{\alpha}}=h_{\alpha}$, for every $\alpha\in I$. Then,
 \[(f|_{S_{\alpha}})|_{S_{\alpha}\cap S_{\beta}}=(h|_{S_{\alpha}})|_{S_{\alpha}\cap S_{\beta}}=(h|_{\beta})|_{S_{\alpha}\cap S_{\beta}}=(f|_{\beta})|_{S_{\alpha}\cap S_{\beta}}.\]

\noindent
For $i\in S_{\alpha}$ we have that $f(i)=(f|_{\alpha})(i)=h_{\alpha}(i)=h(i)$. Since this happens for every $\alpha\in I$, $f=h$.
\end{proof}
\noindent
Let $\textbf{F}$ be a forest and $i\in \textbf{F}$. Since $\Pprod$ is a presheaf of MTL-algebras, its fiber over $i$ is the set of \emph{germs over $i$} and is written as $\Pprod_{i}$ (c.f. II.5 \cite{MM1992}). Recall that $f,g\in \Pprod(S)$ have the same germ at $i$ if there exists some $R\in \dec{\textbf{F}}$ with $i\in R$, such that $R\subseteq S\cap T$ and $f|_{R}=g|_{R}$. Hence, $\Pprod_{i}$ results to be a ``suitable quotient'' of the MTL-algebra $\Pprod(T)$. By Lemma \ref{Forest Product are sheaves}, $\Pprod_{i}$ can be described as the filtering colimit over those $T\in \dec{\textbf{F}}$ such that $i\in T$, i.e.,

\[ \Pprod_{i}=\underrightarrow{lim}_{ {i\in \textbf{T}}}\; \Pprod(T). \]

\noindent
Thereby, for every $T\in \dec{\textbf{F}}$ the map $\varphi_{T}:\Pprod(T) \rightarrow \Pprod_{i}$ that sends $h\in \Pprod(T)$ in its equivalence class ``\emph{modulo germ at $i$}'' result to be  a surjective morphism of MTL-algebras. We write $[h]_{T}$ for the equivalence class of $h$ in $\Pprod_{i}$.

\begin{lem}\label{fibers of Poset Product}
Let $\textbf{F}$ be a forest and $\{\textbf{M}_{i}\}_{i\in \textbf{F}}$ a collection of MTL-chains. For every $i\in F$, $\Pprod(\downarrow i)\cong \Pprod_{i}$ in $\MTL$.
\end{lem}
\begin{proof}
Let $i\in F$ and consider $\varphi_{\downarrow i}: \Pprod(\downarrow i) \rightarrow \Pprod_{i}$. From the above discussion it is clear that $\varphi_{\downarrow i}$ is surjective. To check that it is injective, let $f,g \in \Pprod(\downarrow i)$ be such that $[f]_{\downarrow i}=[g]_{\downarrow i}$. There exists some $R\in \dec{\textbf{F}}$ with $i\in R$, such that $R\subseteq \downarrow i$ and $f|_{R}=g|_{R}$. Since $\downarrow i$ is the smallest downset to which $i$ belongs, we get that $R=\downarrow i$. Then $f=g$. Hence $\varphi_{\downarrow i}$ is an isomorphism in $\MTL$.
\end{proof}

\begin{coro}\label{Forest Product Sheaf is MTLchain}
Let $\textbf{F}$ be a forest and $\{\textbf{M}_{i}\}_{i\in \textbf{F}}$ a collection of MTL-chains. Then $\Pprod$ is a sheaf of MTL-chains.
\end{coro}
\begin{proof}
Apply Lemmas \ref{fibers of Poset Product} and \ref{chain products are MTL-chains}.
\end{proof}

\noindent
Observe that the same argument used in Example $2$ of \cite{BM2011} can be applied to prove that when the index set is a finite chain, the forest product and the ordinal sum of MTL-algebras coincide. The following result will be relevant for the last part of this paper.

\begin{coro}\label{finite Forest Product Sheaf is MTLchain}
Let $\textbf{F}$ be a finite forest and $\{\textbf{M}_{i}\}_{i\in \textbf{F}}$ a collection of MTL-chains. Then for every $i\in \textbf{F}$, $\Pprod_{i}\cong \bigoplus_{i\leq j} \textbf{M}_{i}$.
\end{coro}
\begin{proof}
If $\textbf{F}$ is a finite forest then $\downarrow j$ is a finite chain for every $j\in \textbf{F}$. From the observed above respect to the forest product of MTL-algebras indexed by a finite chain, we conclude that $\Pprod(\downarrow j)\cong \bigoplus_{i\leq j} \textbf{M}_{i}$, which clearly is an MTL-chain. Therefore, from Lemma \ref{fibers of Poset Product} the result follows.
\end{proof}
\noindent
We can now put together the Lemma \ref{Forest Product are sheaves}, and Corollary \ref{Forest Product Sheaf is MTLchain} in the following statement:
\\

\noindent
\emph{The forest product of MTL-chains is essentially a sheaf of MTL-algebras over an Alexandrov space whose fibers are MTL-chains.}

\section{From finite forest products to MTL-algebras}
\label{section from finite forest products to MTL-algebras}

In this section we show that a wide class of finite MTL-algebras can be represented as finite forest products of finite archimedean MTL-chains. To do so, we begin by showing that there exist a functor $\Vuelta$ from the category of finite labeled forests to the category of finite MTL-algebras. Moreover, we will prove that the functor $\Vuelta$ is left adjoint to the functor $\Ida$ and the counit of such adjunction is an isomorphism. It is worth to mention that this result is strongly based in the characterization of the join irreducible elements of a finite forest product of finite archimedean MTL-chains.
\\

\noindent
In general, the unit of the adjoint pair $\Ida \dashv \Vuelta$ is not an isomorphism. In subsection \ref{The duality theorem} we present a class of finite MLT-algebras which solves that problem. Finally, in subsection \ref{An explicit description of finite forest products} we give a simple description of the forest product of finite MTL-chains in terms of ordinal sums and direct products.
\\

\noindent
Let $l:\textbf{F}\rightarrow \Skel$ and $m:\textbf{G}\rightarrow \Skel$ be finite labeled forests. If $(\varphi, \mathcal{F}):l\rightarrow m$ is a morphism of finite labeled forests (see Section \ref{section Finite labeled forests}) then $\varphi: F\rightarrow G$ is a p-morphism and $\mathcal{F}=\{f_{i}\}_{i\in F}$ is a family of morphisms $f_{i}:(m\circ \varphi)(i)\rightarrow l(i)$ of MTL-algebras. Recall that a morphism of posets is a p-morphism if and only if it is open respect to the Alexandrov topologies of the involved posets, so since $F\in \dec{\textbf{F}}$ it follows that $\varphi(F)\in \dec{\textbf{G}}$. From Lemma \ref{forest product is a MTL-algebra}, we get that $\bigotimes_{k\in \varphi(F)}m(k)$ is an MTL-algebra. Notice that $m\circ \varphi:\textbf{F}\rightarrow \Skel$ is a finite labeled forest so we can consider the forest product $\bigotimes_{i\in \textbf{F}}(m\circ \varphi)(i)$. Since $\bigcup_{k\in \varphi(F)}m(k)=\bigcup_{i\in \textbf{F}}(m\circ \varphi)(i)$, we define, for every $h\in \bigotimes_{k\in \varphi(F)}m(k)$, the map $\gamma: \bigotimes_{k\in \varphi(F)}m(k) \rightarrow \bigotimes_{i\in \textbf{F}}(m\circ \varphi)(i)$ as the composite
\begin{displaymath}
\xymatrix{
F \ar[r]^-{\varphi} \ar @/_1pc/ [rr]_-{\gamma(h)} & \varphi(F)\ar[r]^-{h} & \bigcup_{i\in \textbf{F}}(m\circ \varphi)(i)
}
\end{displaymath}

\begin{lem}\label{First igredient}
The map $\gamma$, defined above, is a morphism of MTL-algebras.
\end{lem}
\begin{proof}
In order to check that $\gamma$ is well defined, take $h\in \bigotimes_{k\in \varphi(F)}m(k)$ and consider $i\in F$ such that $\gamma(h)(i)\neq 0_{\varphi(i)}$. Assume $j<i$ in $F$. By the definition of $\gamma$, we get that $h(\varphi(i))\neq 0_{\varphi(i)}$. From the monotonicity of $\varphi$, it follows that $\varphi(j)<\varphi(i)$. Then by assumption, when we have that $h(\varphi(j))=1$, and consequently $\gamma(h)(i)=1$. By Definition \ref{Poset Product}, we have that $\gamma(h)\in  \bigotimes_{i\in \textbf{F}}(m\circ \varphi)(i)$. The proof of the fact that $\gamma$ is an homomorphism is straightforward.
\end{proof}
\noindent
Notice, in addition that the family $\mathcal{F}$ induces a map $\alpha: \bigotimes_{i\in \textbf{F}}(m\circ \varphi)(i) \rightarrow \bigotimes_{i\in \textbf{F}}l(i)$ defined as $\alpha(g)(i)=f_{i}(g(i))$, for every $i\in \textbf{F}$.

\begin{lem}\label{Second igredient}
The map $\alpha$, defined above, is a morphism of MTL-algebras.
\end{lem}
\begin{proof}
Let $g\in \bigotimes_{i\in \textbf{F}}(m\circ \varphi)(i)$ be such that $\alpha(g)(i)\neq 1$. Let $j\in \textbf{F}$ be such that $i<j$. Since, $f_{i}(g(i))\neq 1$, we get that, by Lemma \ref{Characterization of Monomorphisms} $g(i)\neq 1$. Hence, by assumption $g(j)=0_{\varphi(j)}$. Thereby $\alpha(g)(j)=f_{j}(g(j))=0_{j}$, and by $(2)$ of Lemma \ref{Equivalent forms of forest product}, we have that $\alpha(g)\in \bigotimes_{i\in \textbf{F}}l(i)$. The proof of the fact that $\alpha$ is an homomorphism is straightforward.
\end{proof}
\noindent
The Lemmas \ref{First igredient} and \ref{Second igredient} allows us to consider the following composite of morphisms of MTL-algebras:

\begin{displaymath}
\xymatrix{
\Pprod_{m}(\textbf{G}) \ar[r]^-{\beta} & \ \bigotimes_{k\in \varphi(F)}m(k) \ar[r]^-{\gamma} &  \bigotimes_{i\in \textbf{F}}(m\circ \varphi)(i) \ar[r]^-{\alpha} & \Pprod_{l}(\textbf{F})
}
\end{displaymath}
\noindent
where $\Pprod_{m}(G)=\bigotimes_{k\in \textbf{G}}m(k)$, $\Pprod_{l}(\textbf{F})=\bigotimes_{i\in \textbf{F}}l(i)$ and $\beta: \Pprod_{m}(G) \rightarrow \bigotimes_{k\in \varphi(F)}m(k)$ is the restriction of $\Pprod_{m}(G)$ to $\varphi(\textbf{F})$.

\begin{theo}
	\label{functor from forests to fMTL}
 The assignments $l\mapsto \Pprod_{l}(F)$ and $(\varphi, \mathcal{F})\mapsto \alpha \gamma \beta$ define a contravariant functor \[\Vuelta : \fLF \rightarrow \fMTL.\]
\end{theo}
\begin{proof}
Let $l:\textbf{F}\rightarrow \Skel$, $m:\textbf{G}\rightarrow \Skel$ and $n:\textbf{H}\rightarrow \Skel$ be finite labeled forests, and $(\varphi, \mathcal{F}):l\rightarrow m$ and $(\psi, \mathcal{G}):l\rightarrow m$ be morphism of finite labeled forests. Let \[\mathcal{M}=\{f_{i}g_{\varphi(i)}:n(\psi(\varphi (i)))\rightarrow l(i) \mid i\in \textbf{F} \}.\]
\noindent
Consider $s\in \Vuelta(n)$ and $i\in \textbf{F}$. Then from

\begin{eqnarray*}
\nonumber \Vuelta[(\psi, \mathcal{G})(\varphi, \mathcal{F})](s)(i)=\Vuelta(\psi \varphi, \mathcal{M})(s)(i) \nonumber\\ =(f_{i}g_{\varphi(i)})[s(\psi(\varphi(i)))] \nonumber\\ =f_{i}[g_{\varphi(i)}(s(\psi(\varphi(i)))] \nonumber\\ =f_{i}[\Vuelta(\psi, \mathcal{G})(s)(\varphi(i))] \nonumber\\ =\Vuelta(\varphi, \mathcal{F})[\Vuelta(\psi, \mathcal{G})(s)](i) \nonumber\\ =[\Vuelta(\varphi, \mathcal{F})\Vuelta(\psi, \mathcal{G})](s)(i)
\end{eqnarray*}
\noindent
we conclude that $\Vuelta[(\psi, \mathcal{G})(\varphi, \mathcal{F})]=\Vuelta(\varphi, \mathcal{F})\Vuelta(\psi, \mathcal{G})$. Since $id_{l}=(id_{F}, \mathcal{I})$, where $\mathcal{I}$ is the family formed by the identities of $\{l(i)\}_{i\in \textbf{F}}$, it is clear from the definition of $\Vuelta$ that $\Vuelta(id_{l})=id_{\Vuelta(l)}$.
\end{proof}

\begin{lem}\label{characterizationidempotentsposetproduct}
Let $l:\textbf{F}\rightarrow \Skel$ be a finite labeled forest. Then $h\in \mathcal{I}(\Pprod_{l}(\textbf{F}))$ if and only if $C_{h}=\emptyset$.
\end{lem}
\begin{proof}
Observe that $h\in \mathcal{I}(\Pprod_{l}(\textbf{F}))$ if and only if $h(i)^{2}=h(i)$, for every $i\in F$, which is equivalent to say that $h(i)\in \mathcal{I}(l(i))$. Since $l(i)$ is an arquimedean MTL-chain, by Proposition \ref{finite archimedean MTL-chains}, the only possible case is $h(i)=0_{i}$ or $h(i)=1$. This concludes the proof.
\end{proof}
\noindent
Let $\textbf{F}$ be a finite forest and $S\subseteq F$. We write $Max(S)$ for the maximal elements of $S$.

\begin{lem}\label{descriptionofXfilters}
Let $l:\textbf{F}\rightarrow \Skel$ be a finite labeled forest and $S\in \dec{F}$, then \[\filter{S}=\{h\in \Pprod_{l}(\textbf{F}) \mid h(i)=1,\; \text{for every}\; i\in Max(S)\}\]
\end{lem}
\begin{proof}
Let $h\in \Pprod_{l}(\textbf{F})$ and suppose $ h(i)=1$ for every $i\in Max(S)$. If $j\in S$, there exists some $i\in Max(S)$ such that $j\leq i$. Since $h(i)\neq 0_{i}$, $h(j)=1$ and $h\in \filter{S}$. The other inclusion is straightforward.
\end{proof}
\noindent
Recall that, from Corollary \ref{idempotentdeterminesfilters} there exist a unique $h_{S}\in \mathcal{I}(\Pprod_{l}(\textbf{F}))$ such that $\filter{S}=\uparrow h_{S}$. As a direct consequence of Lemmas \ref{characterizationidempotentsposetproduct} and \ref{descriptionofXfilters} we obtain the following result.

\begin{lem}\label{descriptionofidempotentXfilter}
Let $l:\textbf{F}\rightarrow \Skel$ be a finite labeled forest and $S\in \dec{F}$, then \begin{displaymath}
h_{S}(j)= \left\{ \begin{array}{ll}
             1, &  j\leq i, \; \text{for some}\; i\in Max(S)\\
             0_{j}, &  otherwise
             \end{array}
   \right.
\end{displaymath}
\noindent
\end{lem}

\begin{coro}\label{descriptionofprincipalXfilters}
Let $l:\textbf{F}\rightarrow \Skel$ be a finite labeled forest. The following holds for every $i\in F$:
\begin{itemize}
\item[1.] $\filter{\downarrow i}$ is a prime filter of $\Pprod_{l}(\textbf{F})$,
\item[2.] $\filter{\downarrow i}=\{h\in \Pprod_{l}(\textbf{F}) \mid h(i)=1\}$,
\item[3.] The map \begin{displaymath}
h_{\downarrow i}(j)= \left\{ \begin{array}{ll}
             1, &  j\leq i\\
             0_{j}, &  otherwise
             \end{array}
   \right.
\end{displaymath} is a non zero join irreducible element of $\Pprod_{l}(\textbf{F})$.
\end{itemize}
\end{coro}
\begin{proof}
A direct consequence of Lemmas \ref{Chains determines prime filters of poset products}, \ref{descriptionofXfilters}, \ref{descriptionofidempotentXfilter} and Corollary \ref{joinirreducibledeterminesprimefilters}.
\end{proof}

\begin{lem}\label{descriptionofjoinirreducibles}
Let $l:\textbf{F}\rightarrow \Skel$ be a finite labeled forest and $h\in \mathcal{I}(\Pprod_{l}(\textbf{F}))$, then $h\in \JI{\Pprod_{l}(\textbf{F})}$ if and only if $h^{-1}(1)$ is a chain.
\end{lem}
\begin{proof}

Let us assume that $h^{-1}(1)$ is a chain, then, since $F$ is finite, $h^{-1}(1)=\downarrow i$ for some $i\in F$. Suppose that there are $g,f\in \mathcal{I}(\Pprod_{l}(\textbf{F}))$ such that $h=g\vee f$, then $h(k)=g(k)\vee f(k)$, for every $k\in F$. If $k\leq i$, we get that $g(k)\vee f(k)=1$. Since $l(k)$ is really local, by Lemma \ref{really local and prelinear are chains},  $g(k)=1$ or $f(k)=1$. Consequently, $g(k)=h(k)$ or $f(k)=h(k)$. If $h(k)=0_{k}$, the result follows, since $0_{k}$ is join irreducible in $l(k)$. Hence $h$ is join irreducible. On the other hand, suppose that $h\in \JI{\Pprod_{l}(\textbf{F})}$. If $h^{-1}(1)$ is not a chain, there exist $i,j\in F$ not comparables such that $h(i)=h(j)=1$. Let us consider the following functions:
\begin{center}
\begin{tabular}{ccc}
$g(k)=\left\{ \begin{array}{lc}
             h(k), &  \text{if}\; k\neq i  \\
             \\ 0_{i},  &  \text{otherwise} \\
             \end{array}
   \right.$ &  & $f(k)=\left\{ \begin{array}{lc}
             h(k), &  \text{if}\; k\neq j  \\
             \\ 0_{j},  & \text{otherwise} \\
             \end{array}
   \right.$
\end{tabular}
\end{center}
\noindent
From Lemma \ref{characterizationidempotentsposetproduct}, it follows that $g,f\in \mathcal{I}(\Pprod_{l}(\textbf{F}))$. Thereby, $h=g\vee f$, which is in contradiction with the assumption.

\end{proof}

\begin{lem}\label{forest is isomorphic to join irreducibles forest product}
Let $l:\textbf{F}\rightarrow \Skel$ be a finite labeled forest. There is a poset isomorphism between $\textbf{F}$ and $\JI{\Pprod_{l}(\textbf{F})}$.
\end{lem}
\begin{proof}
Let us to consider $\mu :\JI{\Pprod_{l}(\textbf{F})} \rightarrow F$, defined as $\mu(h)=max\; h^{-1}(1)=i_{h}$. From Lemma \ref{descriptionofjoinirreducibles}, it follows that $\mu$ is well defined and is injective. To verify that $\mu$ is surjective, take $i\in F$ and define \begin{displaymath}
h_{i}(j)= \left\{ \begin{array}{ll}
             1, &  j\leq i\\
             0_{j}, &  otherwise.
             \end{array}
   \right.
\end{displaymath}
\noindent
From Lemma \ref{descriptionofjoinirreducibles}, it follows that $h_{i}\in \JI{\Pprod_{l}(\textbf{F})}$. It is clear that $\mu (h_{i})=i$. In order to check the monotonicity of $\mu$, let us suppose that $h\leq g$, for $h,g\in \JI{\Pprod_{l}(\textbf{F})}$. From Corollary \ref{descriptionofprincipalXfilters}.3, we have that $h^{-1}(1)\subseteq g^{-1}(1)$, so $i_{h}\leq i_{g}$ and consequently $\mu (h)\leq \mu(g)$. The monotonicity of $\mu^{-1}$ is straightforward.
\end{proof}

\noindent
Let $l:\textbf{F}\rightarrow \Skel$ be a finite labeled forest. For every $i\in F$ we write $h_{i}$ for the map of $3.$ in Corollary \ref{descriptionofprincipalXfilters}.3. Let us to consider the assignment $\varphi_{l}:F\rightarrow \JI{\Pprod_{l}(\textbf{F})}$, defined as $\varphi_{l}(i)=h_{i}$.

\begin{lem}\label{pmorphismofcounit}
The assignment $\varphi_{l}$ is a p-morphism.
\end{lem}
\begin{proof}
The monotonicity of $\varphi_{l}$ follows from Corollary \ref{descriptionofprincipalXfilters}.3. On the other hand, take $i\in F$ and suppose that $g\leq h_{i}$. Thus, $h(i)=1$ implies that $g(i)=1$ and, due to Corollary \ref{descriptionofprincipalXfilters}.2 we get that $g\in X_{\downarrow i}$. Therefore, $h_{i}\leq g$. In consequence, $g=\varphi_{l}(i)$, which was our aim.
\end{proof}

\begin{lem}\label{pmorphismofcounit is an iso}
The p-morphism $\varphi_{l}$ is an isomorphism.
\end{lem}
\begin{proof}
We first prove the injectivity of $\varphi_{l}$. Let $i,j\in F$ be such that $\varphi_{l}(i)=\varphi_{l}(j)$. Then $h_{i}=h_{j}$. In particular $h_{i}(i)=h_{j}(i)$ and $h_{j}(j)=h_{i}(j)$, which, by definition of $h_{i}$ and $h_{j}$ means that $i\leq j$ and $j\leq i$. The surjectivity of $\varphi_{l}$ follows from Lemma \ref{descriptionofjoinirreducibles}. Finally, we verify that $\varphi_{l}^{-1}$ is also a p-morphism. To do so, notice that for every $h\in \Pprod_{l}(\textbf{F})$, $\varphi_{l}^{-1}(h)$ is just the $i\in F$ described in Lemma \ref{descriptionofjoinirreducibles}. We will denote such element as $i_{h}$. Let us suppose that $j\leq \varphi_{l}^{-1}(h)$, then $h(j)=h(i_{h})=1$ so, for every $k\leq j\leq i_{h}$ we get that $h(k)=h(j)=1$. Take
\begin{displaymath}
g(k)= \left\{ \begin{array}{ll}
             1, &  k\leq j\\
             0_{k}, &  otherwise
             \end{array}
   \right.
\end{displaymath}
It is clear that $g\in \Pprod_{l}(\textbf{F})$, $g\leq h$ and $\varphi_{l}^{-1}(j)=g$.
\end{proof}
\noindent
Observe that, from Lemmas \ref{uniqejoinirreducible} and \ref{pmorphismofcounit} we have that for every $i\in F$, there exists a unique $a_{\varphi_{l}(i)}\in \JI{\Pprod_{l}(\textbf{F})}$ such that $a_{\varphi_{l}(i)}\prec \varphi_{l}(i)$. Due to Lemma \ref{labelingfMTL}, $\uparrow a_{\varphi_{l}(i)}/\uparrow \varphi_{l}(i)$ is an archimedian MTL-chain. Let us consider the assignment $\tau_{i}:\uparrow a_{\varphi_{l}(i)}\rightarrow l(i)$, defined as $\tau_{i}(h)=h(i)$. It is clear, from the definition, that $\tau_{i}$ preserves all the binary monoid operations. Moreover, it preserves the residual. If $ a_{\varphi_{l}(i)} \leq f,g$, then for every $j<i$, $f(j)=g(j)=1$ so $f(j)\leq g(j)$, which means that $(f\rightarrow g)(i)=f(i)\rightarrow_{i}g(i)$, and consequently, $\tau_{i}(f\rightarrow g)=\tau_{i}(f)\rightarrow_{i}\tau_{i}(g)$. We have proved the following result:

\begin{lem}\label{family of morphisms}
The function $\tau_{i}$, defined above, is a morphism of MTL-algebras.
\end{lem}

\noindent
Notice that, from the universal property of quotients in $\MTL$, Lemma \ref{family of morphisms} implies that for every $i\in F$ there exists a unique morphism of MTL-algebras $f_{i}:(\uparrow a_{\varphi_{l}(i)})/(\uparrow \varphi_{l}(i))\rightarrow l(i)$ such that the diagram
\begin{displaymath}
\xymatrix{
\uparrow (a_{\varphi_{l}(i)}) \ar[r] \ar[dr]_-{\tau_{i}} & \uparrow (a_{\varphi_{l}(i)})/\uparrow (\varphi_{l}(i)) \ar[d]^-{f_{i}} \\
 & l(i)
}
\end{displaymath}
\noindent
commutes.

\begin{lem}\label{family of morphisms are isos}
For every $i \in F$, $f_{i}:(\uparrow a_{\varphi_{l}(i)})/(\uparrow \varphi_{l}(i))\rightarrow l(i)$ is an isomorphism.
\end{lem}
\begin{proof}
To prove the inyectivity of $f_{i}$, supppose that $f_{i}([f])=f_{i}([g])$. Then $f(i)=g(i)$. Since $ a_{\varphi_{l}(i)} \leq f,g$, for every $j<i$, $f(j)=g(j)=1$. Hence, $f(i)\leq g(i)$, for every $i\leq j$, which is equivalent to say that $[f]=[g]$ in $(\uparrow a_{\varphi_{l}(i)})/(\uparrow \varphi_{l}(i))$. To check the surjectivity, take $x\in l(i)$. Define
\begin{displaymath}
h(k)= \left\{ \begin{array}{ll}
             1, &  k\leq i\\
             x, &  k=i \\
             0_{k}, & otherwise
             \end{array}
   \right.
\end{displaymath}
It is clear that $f_{i}([h])=x$.
\end{proof}
\noindent
Let $\Ida$ and $\Vuelta$ be the functors from Theorems \ref{functor from fMTL to root systems} and \ref{functor from forests to fMTL}, respectively.

\begin{lem}\label{counit is an isomorphism}
Let $l:\textbf{F}\rightarrow \Skel$ be a finite labeled forest. Then $\Ida \circ \Vuelta$ is isomorphic to the identity functor in $\fLF$.
\end{lem}
\begin{proof}
A direct application of Lemmas \ref{pmorphismofcounit is an iso} and \ref{family of morphisms are isos}.
\end{proof}

\begin{prop}\label{adjunction for finite labeled trees}
The functor $\Ida$ is left adjoint to the functor $\Vuelta$. Moreover, $\Ida$ is full and faithful.
\end{prop}
\begin{proof}
Let $M$ be a finite MTL-algebra, $l:\textbf{F}\rightarrow \Skel$ be a finite labeled forest and $f:M\rightarrow \Pprod_{l}(\textbf{F})$ be a morphism of MTL-algebras. By Lemma \ref{assignation p morphism}, there exists a unique p-morphism $f^{\ast}: \JI{\Pprod_{l}(\textbf{F})} \rightarrow \JI{M}$. Let $\varphi:F\rightarrow \JI{M}$ be defined as $\varphi=f^{\ast}\mu^{-1}$, where $\mu$ is the isomorphism between $F$ and $\JI{\Pprod_{l}(\textbf{F})}$ given in Lemma \ref{forest is isomorphic to join irreducibles forest product}. It is clear that $\varphi$ is a p-morphism. On the other hand, if we write $h_{i}$ for $\mu^{-1}(i)$, then from Lemma \ref{labelingMTLmorphisms}, it follows that for every $i\in F$ there exists a morphism of MTL-algebras $\overline{f}_{h_{i}}:\uparrow a_{f^{\ast}(h_{i})}\rightarrow \uparrow a_{h_{i}}$ which determines a unique morphism of MTL-algebras $f_{h_{i}}:(\uparrow a_{f^{\ast}(h_{i})})/(\uparrow f^{\ast}(h_{i}))\rightarrow (\uparrow a_{h_{i}})/ (\uparrow h_{i})$. From Lemma \ref{family of morphisms are isos} we get that $(\uparrow a_{h_{i}})/ (\uparrow h_{i} \cong l(i))$. Let us to consider $f_{i}:(\uparrow a_{f^{\ast}(h_{i})})/(\uparrow h_{i})\rightarrow l(i)$ as the composition of $f_{h_{i}}$ with the isomorphism of Lemma \ref{family of morphisms are isos}. This concludes the proof.

\end{proof}

\subsection{The duality theorem}
\label{The duality theorem}

In this subsection we present a duality theorem between the class of representable finite MTL-algebras and finite labeled forests. To do so, we restrict the results previously obtained in this section to the class of representable finite MTL-algebras.

\begin{defi}\label{local unit}
Let $M$ be a finite MTL-algebra. An element $e\in \mathcal{I}(M)^{\ast}$ is said to be a \emph{$\GoodId$} if for every $x \leq e$, $ex=x$.
\end{defi}

\begin{lem}\label{equivalent to local  unit}
Let $M$ be a finite MTL-algebra and $e\in \mathcal{I}(M)^{\ast}$. The following are equivalent:
\begin{itemize}
\item[1.] $e$ is a $\GoodId$.
\item[2.] $ey=e\wedge y$, for every $y\in M$.
\end{itemize}
\end{lem}
\begin{proof}
Let us assume that $e$ is a {$\GoodId$} and $y\in M$. Since $e\wedge y\leq y$, $e(e\wedge y)\leq ey$. Hence $e\wedge y\leq ey$, because $e\wedge y\leq e$ and $e(e\wedge y)=e\wedge y$. The other case is straghtforward.
\end{proof}

\begin{defi}\label{goodidempotents}
A finite MTL-algebra $M$ is said to be ${\Appropiate}$ if every non zero idempotent satisfies any of the equivalent conditions of Lemma \ref{equivalent to local  unit}.
\end{defi}
\noindent
Let $M$ be a representable MTL-algebra. For the rest of this section we will write $F_{M}$ to denote $\JI{M}$.

\begin{rem}\label{quotients are intervals}
Observe that as a direct consequence of Definition \ref{local unit}, it follows that for every $e\in F_{M}$, $\uparrow a_{e}/\uparrow e\cong [a_{e},e]$.
\end{rem}

\noindent
Let $M$ be an representable MTL-algebra and $m\in Max(F_{M})$. In what follows we will denote the set $(\downarrow m) \cap F_{M}$ simply by $(\downarrow m)$.

\begin{lem}\label{surjectiveunit}
Let $M$ be a representable MTL-algebra and $m\in Max(F_{M})$. Then, for every $x\in M$ there exists some $e\in \downarrow m$ such that $a_{e}\leq x$.
\end{lem}
\begin{proof}
Suppose that there exist $y>0$ such that for every $e\in \downarrow m$ it holds that $a_{e}\nleq y$. In particular, if $n=min (\downarrow m)$, it follows that $a_{n}=0$; so $0\nleq x$, which is clearly absurd.
\end{proof}

\begin{lem}\label{maximal fibers coincide with quotients}
For every representable MTL-algebra $M$ and $m\in Max(F_{M})$, $M/(\uparrow m) \cong   \bigoplus_{e\leq m} [a_{e},e]$.
\end{lem}
\begin{proof}
In order to simplify the notation for this proof, we write $B$ for $\bigoplus_{e\leq m} [a_{e},e]$. Let us consider the map $f:M\rightarrow B$, defined as $f(x)=x\wedge m$. It is clear that $f$ is a surjective morphism of MTL-algebras such that $f(m)= 1 \in B$. We stress that $B$ has the universal property of $M/\uparrow m$. To prove it, let $g:M\rightarrow E$ be a morphism of MTL-algebras such that $g(m)=1_{E}$. For every $z\in B$ define $\psi(z)=g(z)$. It easily follows that $\psi\circ f=g$, and since $f$ is surjective, then $\psi$ is unique. Hence $M/(\uparrow m) \cong   B$.
\end{proof}

\noindent
Let $M$ be a representable MTL-algebra. Recall that $\Ida(M)=l_{M}:F_{M}\rightarrow \Skel$ is a finite labeled forest, so, from Lemma \ref{forest is isomorphic to join irreducibles forest product}, it follows that $F_{M}\cong \JI{\Pprod_{l_{M}}(F_{M})}$. As a consequence, we get that $\bigoplus_{e\leq m} [a_{e},e]\cong \Pprod_{l_{M}}(\downarrow m)$. Explicitly, such assignment is defined for every $z\in \bigoplus_{e\leq m} [a_{e},e]$ by:

\begin{equation}
h_{z}(c)= \left\{ \begin{array}{ll}
             m, &  c\leq e_{z}\\
             z, &  c=e_{z} \\
             a_{c}, & c>e_{z}
             \end{array}
   \right. \label{equation3}
\end{equation}
\noindent
Here $e_{z}$ is the unique idempotent join irreducible below $m$ such that $a_{e_{z}}\leq z\leq e_{z}$.
\\

\noindent
Observe that $F_{M}=\bigcup_{m\in F_{M}} (\downarrow m)$. Hence the family $\mathcal{R}=\{\downarrow m\}_{m\in M}$ is a covering for $F_{M}$. Let $f_{m}:M\rightarrow \Pprod_{l_{M}}(\downarrow m)$ be defined as $f_{m}(x)=h_{x\wedge m}$.

\begin{lem}\label{algebra induces compatible family}
Let $M$ be a representable MTL-algebra. For every $x\in M$, the family $\{f_{m}(x)\}_{m\in Max(F_{M})}$ is a matching family for the covering $\mathcal{R}$.
\end{lem}
\begin{proof}
Let $m,n\in Max(F_{M})$. Since $F_{M}$ is a forest, $(\downarrow m) \cap (\downarrow n) =\emptyset$ or $(\downarrow m) \cap (\downarrow n) \neq \emptyset$. In the first case, the result holds, because by Lemma \ref{Forest Product are sheaves}, $\Pprod_{l_{M}}$ is a sheaf. In the second case, there exists an $e\in F_{M}$ with $e\leq m,n$ such that $(\downarrow m)\; \cap (\downarrow n)=(\downarrow e)$. Observe that
\[(x\wedge m)\wedge e=x\wedge (m\wedge e)= x\wedge e= x\wedge (n\wedge e)= (x\wedge n)\wedge e.\]
Then, from the descripiton of $h_{x\wedge m}$ and $h_{x\wedge n}$ of equation (\ref{equation3}), we obtain that $h_{x\wedge m}|_{\downarrow e}=h_{x\wedge n}|_{\downarrow e}$.
\end{proof}

\noindent
Recall that Lemma \ref{Forest Product are sheaves} states that $\Pprod_{l_{M}}$ is a sheaf so, from Lemma \ref{algebra induces compatible family}, we obtain that every $x\in M$ determines an amalgamation $h_{x}$ for the family $\{f_{m}(x)\}_{m\in Max(F_{M})}$. This fact allows us to consider the assignment $f_{M}:M\rightarrow \Pprod(F_{M})$, defined as $f(x)=h_{x}$. Observe that by construction $f$ is a morphism of MTL-algebras.

\begin{lem}\label{counit is an isomorphism}
For every representable MTL-algebra $M$, the assignment $f_{M}$ is an isomorphism.
\end{lem}
\begin{proof}
Only remains to check that $f_{M}$ is bijective. To prove the injectivity of $f_{M}$, supppose $h_{x}=h_{y}$, then, since $h_{x}$ and $h_{y}$ are the amalgamations of the families $\{f_{m}(x)\}_{m\in Max(F_{M})}$ and $\{f_{m}(y)\}_{m\in Max(F_{M})}$, respectively, it follows that $h_{x\wedge m}=h_{x\wedge m}$ for every $m\in  Max(F_{M})$. Then, from equation (\ref{equation3}), it follows that $x\wedge m= y\wedge m$, for every $m\in  Max(F_{M})$. Then, $\bigvee_{m\in  Max(F_{M})}(x\wedge m)=\bigvee_{m\in  Max(F_{M})}(y\wedge m)$. Thereby, $x\wedge \bigvee_{m\in  Max(F_{M})} m=y\wedge \bigvee_{m\in  Max(F_{M})} m$. Since $\bigvee_{m\in  Max(F_{M})} m=1$, $x=y$.
\\
Finally, to prove the surjectiviy of $f_{M}$, let $h\in \Pprod(F_{M})$. Then, $h|_{\downarrow m}\in \Pprod(\downarrow m)$ for every $m\in Max(F_{M})$. Since $\Pprod(\downarrow m)\cong \bigoplus_{e\leq m} [a_{e},e]$ we will write $z_{m}$ for the unique element of $\bigoplus_{e\leq m} [a_{e},e]$ which corresponds to $h|_{\downarrow m}$. Observe that for every $m,n\in Max(F_{M})$ we have that
\[(h|_{\downarrow m})|_{\downarrow m \cap \downarrow n}=(h|_{\downarrow n})|_{\downarrow m \cap \downarrow n} \] Thereby $z_{m}\wedge e_{mn}=z_{n}\wedge e_{mn}$, where $e_{mn}$ is the greatest $e\in F_{M}$ below $m$ and $n$. Hence
\begin{equation}
z_{m} \wedge n= (z_{m}\wedge e_{mn}) \wedge (m\wedge n)= (z_{n}\wedge e_{mn}) \wedge (m\wedge n)=z_{n} \wedge m,  \label{equation4}
\end{equation}
since $z_{m}\leq m$ and $z_{n}\leq n$. If we consider $x=\bigvee_{m \in Max(F_{M})}z_{m}$, then applying equation (\ref{equation4}), in the following calculation \[x\wedge n= \bigvee_{m \in Max(F_{M})}(z_{m}\wedge n)=\bigvee_{m \in Max(F_{M})}(z_{n}\wedge m)=z_{n} \wedge (\bigvee_{m\in Max(F(M))} m)=z_{n},\]
\noindent
we obtain that $h_{x}|_{\downarrow n}=h|_{\downarrow n}$ for every $n\in Max(F_{M})$. Thereby, since $h$ is the amalgamation of the family $\{h|_{\downarrow m}\}_{m\in Max(F_{M})}$ and $\Pprod_{l_{M}}$ is a sheaf, it follows that $h_{x}=h$.
\end{proof}
\noindent
Write $\GfMTL$ for the category of representable finite MTL-algebras. Let $\ufLF$ be the subcategory of $\fLF$ whose objects are the finite labeled forest such that their poset product is a representable MTL-algebra. Let us write $\Ida^{\ast}$ for the restriction of the functor $\Ida$ to the category $\GfMTL$ and $\Vuelta^{\ast}$ for the restriction of the functor $\Vuelta$ to $\ufLF$. From Proposition \ref{adjunction for finite labeled trees}, it follows that $\Ida^{\ast}\dashv \Vuelta^{\ast}$ and that the unit is an isomorphism. Notice that the assigment $f_{M}:M\rightarrow \Pprod(F_{M})$ is the counit of the latter adjunction and, by Lemma \ref{counit is an isomorphism}, it is an isomorphism. We have proved the main result of this paper:

\begin{theo}\label{dualitytheorem}
The categories $\GfMTL$ and $\ufLF$ are dually equivalent.
\end{theo}

\subsection{An explicit description of finite forest products}
\label{An explicit description of finite forest products}

The aim of this section is to bring a characterization of the forest product of finite MTL-algebras in terms of ordinal sums and direct products. Unlike the rest of this work, the methods used in this part are completely recursive. Finally, we recall that, along this section, the symbol $\oplus$ will be used indistinctly, to denote the ordinal sum of posets and the ordinal sum of MTL-algebras.
\\

\noindent
In \cite{A2012}, Aguzzoli suggest that every finite forest can be built recursively. We adapt those ideas in the following definition:

\begin{defi}\label{Recursive definition of finite trees}
The set of finite forest $\fFor$ is the smallest colection of finite posets satisfying the following conditions:
\begin{itemize}
\item[\textbf{(F1)}] $\textbf{1}\in \fFor$,
\item[\textbf{(F2)}] If $\textbf{F}\in \fFor$, then $\textbf{1} \oplus \textbf{F}\in \fFor$,
\item[\textbf{(F3)}] If $\textbf{F}_{1},...,\textbf{F}_{m}\in \fFor$ then $\biguplus_{k=1}^{m}\textbf{F}_{k} \in \fFor$.
\end{itemize}
\end{defi}

\noindent
Recall that every finite forest $\textbf{F}$ can be expressed as a finite disjoint union of finite trees. Hence each finite forest can be written as $\textbf{F}=\biguplus_{k=1}^{n}\textbf{T}_{k}$, where $\textbf{T}_{k}$ is a finite tree. We call the family $\{\textbf{T}_{k}\}$ the \emph{family of component trees} of $\textbf{F}$.

\begin{lem}\label{claim1}
Let $\textbf{F}$ be a finite forest and $l:\textbf{F}\rightarrow \Skel$ be a finite labeled forest. Then \[\Pprod_{l}(\biguplus_{k=1}^{n}\textbf{T}_{k})\cong \Pprod_{l_{1}}(\textbf{T}_{1})\times ... \times \Pprod_{l_{n}}(\textbf{T}_{n})\]
where $l_{i}=l|_{\textbf{T}_{i}}$, for every $i=1,...,n$.
\end{lem}
\begin{proof}
Consider $\varphi:\Pprod_{l}(\biguplus_{k=1}^{n}\textbf{T}_{k}) \rightarrow \Pprod_{l_{1}}(\textbf{T}_{1})\times ... \times \Pprod_{l_{n}}(\textbf{T}_{n})$ defined as $\varphi(h)=(h|_{\textbf{T}_{1}},..., h|_{\textbf{T}_{n}})$. Also, consider $\tau: \Pprod_{l_{1}}(\textbf{T}_{1})\times ... \times \Pprod_{l_{n}}(\textbf{T}_{n})\rightarrow \Pprod_{l}(\biguplus_{k=1}^{n}\textbf{T}_{k})$ defined as $\tau(h_{1},...,h_{n})=h$, with $h(i)=h_{j}(i)$ if $i\in \textbf{T}_{j}$. It is clear that $\varphi$ and $\tau$ are well defined morphisms of MTL-algebras and that one is the inverse of the other.
\end{proof}

\begin{lem}\label{claim2}
Let $\textbf{F}$ be a finite forest and $l:\textbf{F}\rightarrow \Skel$ be a finite labeled forest. If $\textbf{F}=1\oplus \textbf{F}_{0}$, where $\textbf{F}_{0}$ is a finite forest, then \[\Pprod_{l}(F)\cong l(\perp)\oplus \Pprod_{l_{0}}( \textbf{F}_{0})\]
where $\perp=Min(F)$ and $l_{0}=l|_{\textbf{T}_{0}}$.
\end{lem}
\begin{proof}
Let us assume $\textbf{F}=1\oplus \textbf{F}_{0}$ and suppose that $h\in \Pprod_{l}(F)$. If $A_{h}=\{i\in F\mid h(i)\neq 1\}$, then either $\perp \in A_{h}$, $\perp \notin A_{h}$ or $h(\perp)=0_{\perp}$. In the first case, from the assumption it follows that $h(j)=0_{j}$ for every $j\in \textbf{F}_{0}$. In the second case, again, by assumption, we obtain that $h(\perp)=1$. The final case implies that $h=0$. Thereby, $h$ represents either an element of $l(\perp)$ (those with $h(\perp)\neq 1$) or an element of $\Pprod_{l_{0}}( \textbf{F}_{0})$ considered as $h=h|_{F_{0}}$. Based on this fact we consider $p:\Pprod_{l}(\textbf{F})\rightarrow l(\perp)\oplus \Pprod_{l_{0}}(\textbf{F}_{0})$, defined for every $h\in \Pprod_{l}(\textbf{F})$ as
\begin{displaymath}
p(h)= \left\{ \begin{array}{ll}
             h(\perp), &  if\; h(\perp)\neq 1\\
             h|_{\textbf{F}_{0}}, & if\; h(\perp)= 1
             \end{array}
   \right.
\end{displaymath}
Let $a\in l(\perp)\oplus \Pprod_{l_{0}}(\textbf{F}_{0})$. Then, $a$ is either an element of $l(\perp)$ or $a\in \Pprod_{l_{0}}(\textbf{F}_{0})$. Let us take $q:l(\perp)\oplus \Pprod_{l_{0}}(\textbf{F}_{0})\rightarrow \Pprod_{l}(\textbf{F})$ as $q(a)=h_{a}$, where $h_{a}(i)=a(i)$, if $i\in \textbf{F}_{0}$ or
\begin{displaymath}
h_{a}(\perp)= \left\{ \begin{array}{ll}
             a, &  if\; a\in l(\perp)\\
             1, & otherwise
             \end{array}
   \right.,
   \textrm{ if } i \notin \textbf{F}_{0}.
\end{displaymath}
It is clear that $p$ and $q$ are well defined morphisms of MTL-algebras such that one is the inverse of the other. This concludes the proof.
\end{proof}

\noindent
Let $\textbf{T}$ be a finite tree, $i\in \textbf{T}$ and consider the \emph{set of covering elements} of $i$:
\[\cov{\textbf{T}}{i}=\{j\in \textbf{T} \mid i\prec j\} \]
\noindent
where $\prec$ denotes the covering relation in posets.

\begin{defi}\label{Inductive construction}
Let $\textbf{T}$ be a finite tree and $l:\textbf{T}\rightarrow \Skel$ be a finite labeled forest. For every $i\in T$ let us define recursively the following MTL-algebra:

\begin{displaymath}
\mtl{\textbf{T}}{i}= \left\{ \begin{array}{ll}
             l(i), &  i\in Max(T)\\
             l(i)\oplus \Pi_{j\in \cov{\textbf{T}}{i}}\mtl{\textbf{T}}{j}, &  i\notin Max(T).
             \end{array}
   \right.
\end{displaymath}
\noindent

\end{defi}

\noindent
In the following, we will write $\mathcal{K}_{l}(\textbf{T})$ for $\mtl{\textbf{T}}{m}$, where $m$ is the last element of $\textbf{T}$.
\\

\noindent
Let $\textbf{F}$ be a finite forest and $\{\textbf{T}_{k}\}$ be its collection of component trees. If $l:\textbf{F}\rightarrow \Skel$ is a finite labeled forest let us consider the MTL-algebra:

\begin{equation}\label{Defining Functor}
\mathcal{K}_{l}(\textbf{F})=\prod_{k=1}^{n} \mathcal{K}_{l}(\textbf{T}_{k})
\end{equation}


\begin{prop}\label{explicit description of forest products}
Let  $l:\textbf{F}\rightarrow \Skel$ be a finite labeled forest. Then $\Pprod_{l}(\textbf{F})\cong \mathcal{K}_{l}(\textbf{F})$.
\end{prop}
\begin{proof}
We prove this Proposition by induction over $\fFor$. If $F=\bf{1}$, the conclusion is trivial. On the other hand, let us suppose that $\textbf{F}=\textbf{1}\oplus \textbf{F}_{0}$, with $\textbf{F}_{0}\in \fFor$ be such that $\Pprod_{l}(\textbf{F}_{0})\cong \mathcal{K}_{l}(\textbf{F}_{0})$. From Lemma \ref{claim2} and the inductive hypothesis, it follows that \[\Pprod_{l}(\textbf{F})\cong l(\perp) \oplus \Pprod_{l_{0}}(\textbf{F}_{0})\cong l(\perp)\oplus \mathcal{K}_{l_{0}}(\textbf{F}_{0}).\] Since $\textbf{F}_{0}$ is a finite forest, $\textbf{F}_{0}=\biguplus_{k=1}^{r}\textbf{T}_{k}^{0}$. Thus by equation (\ref{Defining Functor}): \[\mathcal{K}_{l_{0}}(\textbf{F}_{0})=\prod_{k=1}^{r} \mathcal{K}_{l_{0}}(\textbf{T}_{k}^{0})=\prod_{k=1}^{r}\mtl{\textbf{T}_{k}^{0}}{m_{k}},\] with $m_{k}=Min(\textbf{T}_{k}^{0})$. Since $\cov{\textbf{F}}{\perp}=\{m_{1},...,m_{r}\}$, by Definition \ref{Inductive construction}, we have that $\Pprod_{l}(\textbf{F})\cong l(\perp)\oplus \mathcal{K}_{l_{0}}(\textbf{F}_{0})= \mathcal{K}_{l}(\textbf{F})$. Finally, assume that $\textbf{F}=\biguplus_{k=1}^{n}\textbf{F}_{k}$, with $\textbf{F}_{k}\in \fFor$ be such that $\Pprod_{l}(\textbf{F}_{k})\cong \mathcal{K}_{l}(\textbf{F}_{k})$, for every $i=1,...,n$. Since $\textbf{F}_{k}=\biguplus_{i=1}^{m_{k}}\textbf{T}_{ki}$, with $\{\textbf{T}_{ki}\}$ be the family of component trees of $\textbf{F}_{k}$, then $\textbf{F}=\biguplus_{k=1}^{n} \biguplus_{i=1}^{m_{k}}\textbf{T}_{ki}$. By Lemma \ref{claim1}, we have that

\begin{equation}\label{equationfinal}
\Pprod_{l}(\textbf{F})\cong \Pprod_{l}(\textbf{T}_{11})\times ... \times \Pprod_{l}(\textbf{T}_{1m_{1}})\times ... \times\Pprod_{l}(\textbf{T}_{n1})\times ... \times \Pprod_{l}(\textbf{T}_{nm_{n}}).
\end{equation}

\noindent
Since the direct product of algebras is associative, by Lemma \ref{claim1}, $\Pprod_{l}(\textbf{F}_{k})\cong \prod_{i=1}^{m_{k}}\Pprod_{l}(\textbf{T}_{ki})$. Then, from equation (\ref{equationfinal}), we get that $\Pprod_{l}(\textbf{F})\cong \prod_{k=1}^{n}\Pprod_{l}(\textbf{F}_{k})$. Hence, by inductive hypothesis, we have that $\Pprod_{l}(\textbf{F})\cong \prod_{k=1}^{n}\mathcal{K}_{l}(\textbf{F}_{k})$. From equation (\ref{Defining Functor}), we have that $\mathcal{K}_{l}(\textbf{F}_{k})=\prod_{i=1}^{m_{k}}\mathcal{K}_{l}(\textbf{T}_{ki})$. Since, the family of component trees of $\textbf{F}$ is $\bigcup_{k,i=1}^{n,m_{k}}\{\textbf{T}_{ki}\}$, using again equation (\ref{Defining Functor}), we conclude that $\Pprod_{l}(\textbf{F})\cong \mathcal{K}_{l}(\textbf{F})$.
\end{proof}

\noindent
In the following example we illustrate how to build an MTL-algebra by applying equation (\ref{Defining Functor}) and Definition \ref{Inductive construction}.

\begin{ex}\label{Application of the algorithm}
Let  $l:\textbf{F}\rightarrow \Skel$ be a finite labeled forest and regard the following finite tree $\textbf{F}$:

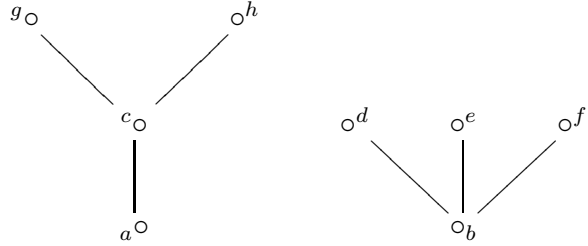
\begin{displaymath}
\xymatrix{
^g {\circ}  \ar@{-}[dr] & & {\circ} ^h \ar@{-}[dl] & & &  &
\\
 & ^c {\circ} \ar@{-}[d] & & {\circ} ^d \ar@{-}[dr] & {\circ} ^e \ar@{-}[d] & {\circ} ^f \ar@{-}[dl] &
\\
 & _a {\circ} & &  & {\circ} _b & &
}
\end{displaymath}
\noindent
As we can see, $\textbf{F}=\textbf{T}_{1} \uplus \textbf{T}_{2}$, where $\textbf{T}_{1}=\{a,c,g,h\}$ and $\textbf{T}_{2}=\{b,d,e,f\}$. Since $Min(\textbf{T}_{1})=\{a\}$, $Min(\textbf{T}_{2})=\{b\}$, from equation (\ref{Defining Functor}), it follows that:
\[ \mathcal{K}_{l}(\textbf{F})= \mathcal{K}_{l}(\textbf{T}_{1}) \times \mathcal{K}_{l}(\textbf{T}_{2})=\mtl{\textbf{T}_{1}}{a} \times \mtl{\textbf{T}_{2}}{b}.\]

\noindent
Observe that $\cov{\textbf{T}_{1}}{a}=\{c\}$ and $\cov{\textbf{T}_{2}}{b}=\{d,e,f\}$. Since $a\notin Max(\textbf{T}_{1})$ and $b\notin Max(\textbf{T}_{2})$, applying the Definition \ref{Inductive construction} to $\mtl{\textbf{T}_{1}}{a}$ and $\mtl{\textbf{T}_{2}}{b}$ respectively, we obtain:
\[ \mathcal{K}_{l}(S)=[l(a)\oplus \mtl{\textbf{T}_{1}}{c}] \times [l(b)\oplus (\mtl{\textbf{T}_{2}}{d} \times \mtl{\textbf{T}_{2}}{e} \times \mtl{\textbf{T}_{2}}{f} )]. \]

\noindent
Since $c\notin Max(\textbf{T}_{1})$ and $\cov{\textbf{T}_{1}}{c}=\{g,h\}$ but $d,e,f\in Max(\textbf{T}_{2})$; again, applying Definition \ref{Inductive construction} to $\mtl{\textbf{T}_{1}}{c}, \mtl{\textbf{T}_{2}}{d}$, $\mtl{\textbf{T}_{2}}{e}$ and $\mtl{\textbf{T}_{2}}{f}$ respectively, we get:
\[ \mathcal{K}_{l}(S)=[l(a)\oplus (l(c)\oplus (\mtl{\textbf{T}_{1}}{g} \times \mtl{\textbf{T}_{1}}{h}))] \times [l(b)\oplus (l(d) \times l(e) \times l(f) )]. \]

\noindent
Finally, since $g,h\in Max(\textbf{T}_{1})$, applying the Definition \ref{Inductive construction} to $\mtl{\textbf{T}_{1}}{g}$ and $\mtl{\textbf{T}_{1}}{h}$, we conclude that:
\[\mathcal{K}_{l}(S)=[l(a)\oplus (l(c)\oplus (l(g) \times l(h)))] \times [l(b)\oplus (l(d) \times l(e) \times l(f) )]. \]
\end{ex}

\section*{Acknowledgements}

This work was supported by CONICET [PIP 112-201501-00412].


\begin{thebibliography}{9}

\bibitem{A2012} S. Aguzzoli. Finite Forests. Their Algebras and Logics.\emph{ Slides for LATD 2012}. Kanazawa, Japan. 2012

\bibitem{ABM2009} S. Aguzzoli, S. Bova and V. Marra. Applications of Finite Duality to Locally Finite Varieties of BL-Algebras, in S. Artemov and A. Nerode (eds.),
Logical Foundations of Computer Science, \emph{Lecture Notes in Computer Science}, 5407 (2009) 1-15.	
	
\bibitem{B2004} M. Busaniche. Decomposition of BL-chains,
\emph{Algebra universalis}, 52 (2004) 519--525.

\bibitem{BM2011} M. Busaniche, and F. Montagna. Chapter VII: Basic Fuzzy Logic and BL-algebras, \emph{Handbook of Mathematical Fuzzy Logic}. Vol I. Studies in Logic. College Publications. 2011.


\bibitem{CMZ2016} J. L. Castiglioni, M. Menni and W. J. Zuluaga Botero. A representation theorem for integral rigs and its applications to residuated lattices,
\emph{Journal of Pure and Applied Algebra}, 220 (10) (2016) 3533--3566.

\bibitem{Cliff1958} A. H. Clifford. Totally ordered commutative semigroups,
\emph{Bull. Amer. Math. Soc.}, 64(6) (1958) 305--316.

\bibitem{DNL2003} A. Di Nola and A. Lettieri. Finite BL-algebras,
\emph{Discrete Math.}, 269 (2003) 93--122.

\bibitem{EG2001} F. Esteva and L. Godo. Monoidal t--norm based Logic: Towards a logic for left--continuous t--noms,
\emph{Fuzzy Sets and Systems}, 124 (2001) 271--288.

\bibitem{HRT2002} J. B. Hart, L. Rafter and C. Tsinakis. The Structure of Commutative Residuated Lattices,
\emph{International Journal of Algebra and Computation}, 12(4) (2002) 509--524 .

\bibitem{H2011} R. Hor\v{c}ik. On the Structure of Finite Integral Commutative Residuated Chains,
\emph{Journal of Logic and Computation}, 21(5) (2011) 717--728.

\bibitem{HM2009} R. Hor\v{c}ik and F. Montagna. Archimedean classes in integral commutative residuated chains,
\emph{Math. Log. Quart.}, 55(3) (2009) 320--336.

\bibitem{KT2006} N. Kehayopulu and M. Tsingelis. Decomposition of Commutative Ordered Semigroups into Arquimedean Components,
\emph{Lobachevskii Journal of Mathematics}, 22 (2006) 27--34.

\bibitem{MM1992} S. MacLane and I. Moerdijk. \emph{Sheaves in Geometry and Logic: A First Introduction to Topos Theory}. Universitext, Springer Verlag (1992).

\bibitem{MUR2008} C. Mure\c{s}an. The Reticulation of a Residuated Lattice,
\emph{Bulletin math\'ematique de la Soci\'et\'e des Sciences Math\'ematiques de Roumanie}, 1 (2008) 47--65.

\bibitem{NEG2005} C. Noguera, F. Esteva and J. Gispert. On some Varieties of MTL-algebras,
\emph{Logic Journal of the IGPL}, 13 (2005) 443--466.

\bibitem{PV2014} M. Petr\'{\i}k and Th. Vetterlein. Algorith to generate the Archimedean, finite negative tomonoids,
\emph{SCIS\&ISIS 2014}, Kitakyushu, Japan, December 3-6, 2014.

\bibitem{SB1981} H. P. Sankappanavar and S. Burris. \emph{A course in universal algebra}, Graduate Texts Math, 78 (1981)

\bibitem{TK1958} T. Tamura and N. Kimura. On decompositions of a commutative semigroup,
\emph{Kodai Math. Sem. Rep. }, 6(4) (1954) 109--112.

\bibitem{TU1999} E. Turunen. BL-algebras of basic fuzzy logic,
\emph{Mathware \& Soft computing}, 6(1) (1999) 49--61.


\bibitem{XX2013} I. A. Xarez and J. J. Xarez. Galois Theories of Commutative Semigroups via Semilattices,
\emph{Theory and Applications of Categories}, 28(33) (2013) 1153--1169.

\bibitem{Z2016} W. Zuluaga Botero. Representation by Sheaves of riRigs (Spanish). PhD Thesis. Universidad Nacional de La Plata. \emph{http://sedici.unlp.edu.ar/handle/10915/54115.} (2016)

\end{thebibliography}
\end{document}